\documentclass[reqno]{amsproc}
\usepackage[utf8]{inputenc}
\usepackage{cite}
\usepackage{xcolor}
\usepackage{color}
\usepackage{calligra}
\usepackage[T1]{fontenc}
\usepackage{latexsym}
\usepackage[all]{xy}
\usepackage{hyperref}
\usepackage{mathrsfs}
\usepackage{amsmath,amsfonts,amssymb,amsxtra}
\numberwithin{equation}{section}
\newtheorem{theorem}{\bf Theorem}[section]

\newtheorem{lem}{\bf Lemma}[section]
\newtheorem{cor}{\bf Corollary}[section]
\newtheorem{remark}{\bf Remark}[section]

\newtheorem{example}{\bf Example}[section]

\newcommand\dv{\mathrm{div}}
\newcommand\tr{\mathrm{tr}}
\newcommand\rc{\mathrm{Ric}}
\usepackage{algorithmic}
\begin{document}

\title[Eigenvalues of fourth-order elliptic operators in divergence form]{Inequalities for eigenvalues of fourth-order elliptic operators in divergence form on complete Riemannian manifolds}
\author[M.C. Araújo Filho]{Marcio C. Araújo Filho$^1$}
\address{$^1$Departamento de Matemática, Universidade Federal de Rondônia, Campus Ji-Paraná, R. Rio Amazonas, 351, Jardim dos Migrantes, 76900-726 Ji-Paraná, Rondônia, Brazil}
\email{$^1$marcio.araujo@unir.br}

\begin{abstract}
We prove some inequalities of Payne–Pólya–Weinberger–Yang type  for eigenvalues of fourth-order elliptic operators  in weighted divergence form on complete Riemannian manifolds which generalizes the corresponding result for the clamped plate problem. We also prove estimates for lower order eigenvalues that contain some of the estimates from the literature. As an application of our results, we obtain eigenvalues estimates for the bi-drifted Cheng-Yau operator.  
\end{abstract}
\maketitle

\section{Introduction and the main results}
Let $(M^n, \langle , \rangle)$ be an $n$-dimensional complete Riemannian manifold and $\Omega\subset M^n$ be a bounded and connected domain with smooth boundary $\partial\Omega$. Let us consider a symmetric positive definite $(1,1)$-tensor  $T$ on $M^n$ and a smooth real-valued function $\eta$ on $M$, so that we define a second-order elliptic differential operator $\mathscr{L}$ in the $(\eta,T)$-divergence form as follows:
\begin{equation}\label{eq1.1}
    \mathscr{L}f :=\dv_\eta (T(\nabla f)) = \dv(T(\nabla f)) - \langle \nabla \eta, T(\nabla f) \rangle,
\end{equation}
where $\dv$ stands for the divergence operator and $\nabla$ is the gradient operator. 

The purpose of this paper is to give some inequalities for eigenvalues of a larger class of fourth-order elliptic operators in divergence form on Riemannian manifolds, which contains some known estimates of the literature. To be more precise, we shall consider the following eigenvalue problem of  fourth order with Dirichlet boundary condition:
\begin{equation}\label{problem1}
    \left\{\begin{array}{ccccc} 
    \mathscr{L}^2  u  &=&  \Gamma u & \mbox{in } & \Omega,\\
     u=\frac{\partial u}{\partial \nu_T}&=&0 & \mbox{on} & \partial\Omega, 
    \end{array} 
    \right.
\end{equation}
where $\frac{\partial u}{\partial \nu_T}=\langle T(\nabla u), \nu \rangle$ with $\nu$ the outward unit normal vector field of $\partial \Omega$. 

It can be seen that $\mathscr{L}^2$ is a formally self-adjoint operator in the space of all smooth real-valued functions $u$ such that $u|_{\partial \Omega}=\frac{\partial u}{\partial \nu_T}|_{\partial \Omega}=0$, with respect to the inner product
\begin{align*}
    \langle\langle u , v \rangle \rangle = \int_\Omega uv dm,
\end{align*}
where $dm = e^{-\eta}d\Omega$ is the weighted volume form on $\Omega$ (see Section~\ref{preliminaries}). Thus the spectrum of Problem~\eqref{problem1} is real and discrete, that is,
\begin{equation*}
    0 < \Gamma_1 \leq \Gamma_2 \leq \cdots \leq \Gamma_k \leq \cdots\to\infty,
\end{equation*}
where each $\Gamma_i$ is repeated according to its multiplicity.

If $\eta$ is a constant and $T$  is the identity operator, immediately from \eqref{eq1.1} we notice that operator $\mathscr{L}$ is the Laplacian operator and Problem~\ref{problem1} becomes the following eigenvalue problem for the {\it Dirichlet biharmonic operator }
\begin{equation}\label{problem2}
    \left\{\begin{array}{ccccc} 
    \Delta^2  u  &=& \Gamma u & \mbox{in } & \Omega,\\
     u=\frac{\partial u}{\partial \nu}&=&0 & \mbox{on} & \partial\Omega, 
    \end{array} 
    \right.
\end{equation}
where $\Delta^2$ is the biharmonic operator on $C^\infty(\Omega)$. Problem~\eqref{problem2} is known as the {\it clamped plate problem} which describes the characteristic vibrations of a clamped plate in elastic mechanics. 
{The clamped plate problem goes back to Lord Rayleigh, who formulated an isoperimetric-type problem for the principal eigenvalue (in 1877). This problem has been settled later, in the 1990s, by Nadirashvili~\cite{Nadirashvili} and Ashbaugh and Banguria~\cite{AshbaughBenguria}, in dimensions two and three, in the Euclidean setting, and Ashbaugh and Laugesen~\cite{AshbaughLaugesen} for higher dimensions. Furthermore, very recently, Kristaly~\cite{Kristaly} handled the problem of clamped plates on Riemannian manifolds with negative curvature.} For the eigenvalue problem of the clamped plate problem, some interesting inequalities have been established at works \cite{Cheng-etal1}-\cite{ChengYang5}, \cite{HileYeh}, \cite{Hook}, \cite{Payne-etal}, and \cite{WangXia}, we present some of them in more detail below. 

In the Euclidean space case, that is $\Omega \subset \mathbb{R}^n$, Payne et al.~\cite{Payne-etal} gave the following estimate for the eigenvalues of Problem~\eqref{problem2}
\begin{equation}\label{Payne-etal-estimate}
    \Gamma_{k+1} - \Gamma_k \leq \frac{8(n+2)}{n^2}\frac{1}{k}\sum_{i=1}^k\Gamma_i.
\end{equation}
Later, Hile and Yeh~\cite{HileYeh} proved the following inequality
\begin{equation*}
    \sum_{i=1}^k\frac{\Gamma_i^{\frac{1}{2}}}{\Gamma_{k+1}-\Gamma_i}\geq \frac{n^2k^{\frac{3}{2}}}{8(n+2)} \Bigg(\sum_{i=1}^k\Gamma_i\Bigg)^{-\frac{1}{2}},
\end{equation*}
which generalizes Inequality~\eqref{Payne-etal-estimate}.
Furthermore, Hook\cite{Hook} also established the following estimate
\begin{equation*}
    \frac{n^2k^2}{8(n+2)}\leq \Bigg(\sum_{i=1}^k\Gamma_i\Bigg)^{\frac{1}{2}}\Bigg(\sum_{i=1}^k\frac{\Gamma_i^{\frac{1}{2}}}{\Gamma_{k+1}-\Gamma_i}\Bigg).
\end{equation*}
Ashbaugh asked, in his survey paper \cite{Ashbaugh}, whether it is possible to obtain inequalities for eigenvalues for Problem~\eqref{problem2} that are similar to well-known  Yang's inequality obtained for the Laplacian case by Yang~\cite{Yang}. The question asked by Ashbaugh had a positive answer given by Cheng and Yang~\cite{ChengYang5}, more precisely they proved the following inequality
\begin{equation*}
    \Gamma_{k+1}-\frac{1}{k}\sum_{i=1}^k\Gamma_i \leq \Big(\frac{8(n+2)}{n}\Big)^{\frac{1}{2}} \frac{1}{k}\sum_{i=1}^k\Big(\Gamma_i(\Gamma_{k+1}-\Gamma_i)\Big)^\frac{1}{2}.
\end{equation*}

Cheng et al.~\cite{Cheng-etal} studied the Problem~\eqref{problem2} when $M$ is an $n$-dimensional submanifold isometrically immersed in a Euclidean space. In fact, they proved that if {\bf H} is the mean curvature vector of this immersion, then
\begin{equation}\label{Cheng-etal-estimate}
    \sum_{i=1}^k(\Gamma_{k+1}-\Gamma_i)^2\leq \frac{1}{n^2}\sum_{i=1}^k(\Gamma_{k+1}-\Gamma_i)\Big((2n+4)\Gamma_i^{\frac{1}{2}}+n^2H_0^2\Big)(4\Gamma_i^{\frac{1}{2}}+n^2H_0^2),
\end{equation}
where $H_0=\sup_\Omega |{\bf H}|$. More recently, on the same configurations, Wang and Xia~\cite{WangXia} proved the following inequality
\begin{align}\label{WangXia-estimate}
    \sum_{i=1}^k(\Gamma_{k+1}-\Gamma_i)^2 \leq& \frac{1}{n}\Bigg\{\sum_{i=1}^k(\Gamma_{k+1}-\Gamma_i)^2\Big[(4+2n)\Gamma_i^{\frac{1}{2}} + n^2H_0^2\Big]\Bigg\}^{\frac{1}{2}}\nonumber\\
    &\times\Bigg\{\sum_{i=1}^k(\Gamma_{k+1}-\Gamma_i)\Big(4\Gamma_i^{\frac{1}{2}}+ n^2H_0^2\Big)\Bigg\}^{\frac{1}{2}},
\end{align}
and, using the Reverse Chebyshev Inequality, they showed that Inequality~\eqref{WangXia-estimate} implies Inequality~\eqref{Cheng-etal-estimate}, cf. \cite[Remark~2.2]{WangXia}. 

If $\eta$ is not necessarily constant and $T=I$, from \eqref{eq1.1} the operator $\mathscr{L}$ becomes the {\it drifted Laplacian operator} $\Delta_\eta (\cdot) = \Delta (\cdot) - \langle \nabla \eta, \cdot \rangle$, and we can rewritten the Problem~\eqref{problem1} as follows
\begin{equation}\label{problem3}
    \left\{\begin{array}{ccccc} 
    \Delta_\eta^2  u  &=& \Gamma u & \mbox{in } & \Omega,\\
     u=\frac{\partial u}{\partial \nu}&=&0 & \mbox{on} & \partial\Omega, 
    \end{array} 
    \right.
\end{equation}
where $\Delta_\eta^2$ is the bi-drifting Laplacian operator on $C^\infty(\Omega)$. 

Recently, when $M$ is an $n$-dimensional submanifold isometrically immersed in a Euclidean space with mean curvature vector ${\bf H}$, Du et al.~\cite{Du-etal} established the following eigenvalues inequality of Problem~\eqref{problem3}
\begin{align}\label{Du-etal-estimate}
    \sum_{i=1}^k(\Gamma_{k+1}-\Gamma_i)^2 \leq& \frac{1}{n}\Bigg\{\sum_{i=1}^k(\Gamma_{k+1}-\Gamma_i)^2\Big[(2n+4)\Gamma_i^{\frac{1}{2}} +4\eta_0\Gamma_i^{\frac{1}{4}} + n^2H_0^2+\eta_0^2\Big]\Bigg\}^{\frac{1}{2}}\nonumber\\
    &\times\Bigg\{\sum_{i=1}^k(\Gamma_{k+1}-\Gamma_i)\Big(4\Gamma_i^{\frac{1}{2}}+4\eta_0\Gamma_i^{\frac{1}{4}} + n^2H_0^2+\eta_0^2 \Big)\Bigg\}^{\frac{1}{2}},
\end{align}
where $H_0=\sup_\Omega |{\bf H}|$ and $\eta_0=\max_{\Bar{\Omega}}|\nabla \eta|$. We can obtain Inequality~\eqref{Du-etal-estimate} as an application of one of our results, cf. Theorem~\ref{Generalization-Duetal-teo} in Section~\ref{concludingremarks}. 

We would like to observe that one can impose restrictions on the weighted Ricci curvature (for example, weighted Ricci curvature nonnegative) to get estimates for the first eigenvalue of Problem~\eqref{problem3}, see e.g. Araújo Filho~\cite{AraujoFilho} and references therein.

Another interesting problem is to found estimates for lower order eigenvalues for Problem~\eqref{problem1}. Some estimates for lower order eigenvalues have been obtained over the years, for Problem~\eqref{problem2} we refer to Cheng et al.~\cite{Cheng-etal1} and their references and for Problem~\eqref{problem3} we refer Du et al.~\cite{Du-etal} and their references.

Some of the inequalities for eigenvalues presented above are included in our results, but before presenting our theorems, we first denote by ${\bf H}_T$ the generalized mean curvature vector associated with $(1, 1)$-tensor $T$, see Section~\ref{preliminaries} for more details. The definition of the generalized mean curvature vector was considered by Grosjean~\cite{Grosjean} and Roth~\cite{Roth2}. Moreover, since the $(1, 1)$-tensor $T$ is symmetric and positive definite and $\Omega$ is a bounded domain in Problem~\ref{problem1}, there are positive real numbers $\varepsilon$ and $\delta$ such that $\varepsilon I \leq T \leq \delta I$ where $I$ is the identity $(1, 1)$-tensor. In this most general setting of Problem~\eqref{problem1}, we apply known techniques to prove our results. 

\begin{theorem}\label{theorem1.1}
Let $\Omega$ be a bounded domain in an $n$-dimensional complete Riemannian manifold $M^n$ isometrically immersed in $\mathbb{R}^m$. Denote by $\Gamma_i$ the $i$-th eigenvalue  of Problem~\eqref{problem1}, then we have
\begin{align}\label{theorem1-estimate1}
    \sum_{i=1}^k&(\Gamma_{k+1}-\Gamma_i)^2\nonumber\\
    \leq& \frac{4}{n\varepsilon}\Bigg\{\sum_{i=1}^k(\Gamma_{k+1}-\Gamma_i)^2\Big[\Big(\sqrt{\delta \Gamma_i^{\frac{1}{2}}}  + \frac{T_0}{2}\Big)^2 +\frac{n\delta \Gamma_i^{\frac{1}{2}}}{2} + \frac{n^2H_0^2+4C_0+2T_0\delta\eta_0}{4} \Big]\Bigg\}^{\frac{1}{2}}\nonumber\\
    &\times\Bigg\{\sum_{i=1}^k(\Gamma_{k+1}-\Gamma_i)\Big[\Big(\sqrt{\delta \Gamma_i^{\frac{1}{2}}}  + \frac{T_0}{2}\Big)^2 + \frac{n^2H_0^2+4C_0+2T_0\delta\eta_0}{4} \Big]\Bigg\}^{\frac{1}{2}},
\end{align}
and
\begin{align}\label{theorem1-estimate2}
    \sum_{i=1}^k(\Gamma_{k+1}-\Gamma_1)^\frac{1}{2}\leq& \frac{1}{\varepsilon}\Bigg[4\Big(\sqrt{\delta \Gamma_1^{\frac{1}{2}}}  + \frac{T_0}{2}\Big)^2 + n^2H_0^2+4C_0+2T_0\delta\eta_0\Bigg]^{\frac{1}{2}}\nonumber\\
    &\times\Bigg[4\Big(\sqrt{\delta \Gamma_1^{\frac{1}{2}}}  + \frac{T_0}{2}\Big)^2 + 2n\delta \Gamma_1^{\frac{1}{2}} + n^2H_0^2+4C_0+2T_0\delta\eta_0\Bigg]^{\frac{1}{2}},
\end{align}
where
\begin{equation}\label{C_0}
    C_0=\sup_\Omega \Big\{\frac{1}{2}\dv (T^2(\nabla \eta)) - \frac{1}{4}|T(\nabla \eta)|^2\Big\},
\end{equation}
$T_0=\sup_{\Omega}|\tr(\nabla T)|$, $\eta_0=\sup_{\Omega}|\nabla \eta|$, $H_0=\sup_\Omega |{\bf H}_T|$ and ${\bf H}_T$ is the generalized mean curvature vector of the immersion.
\end{theorem} 
\begin{remark}
Notice that Theorem~\ref{theorem1.1} is the strengthens Inequality~(1.14) in \cite{WangXia}. In fact, if $T=I$ and $\eta$ is constant we have $\varepsilon=\delta=1, T_0 = C_0 = 0$ and ${\bf H}_T={\bf H}$ is the mean curvature vector of the immersion, consequently the first inequality of Theorem~\ref{theorem1.1} becomes Inequality~\eqref{WangXia-estimate} obtained by Wang and Xia~\cite{WangXia} for Problem~\ref{problem2}.
\end{remark}

We would like to emphasize a geometric interpretation for $C_0$ in \eqref{C_0}, initially observed by Wang~\cite{Wang} and mentioned by Gomes and Araújo Filho\cite{GomesAraujoFilho},
that $C_0$ can be seen as the supremum of the scalar curvature on the warped product $\Omega\times_{e^{-\frac{\eta}{2}}}\mathbb{S}^1$. Moreover, notice that the constant $C_0$ depends only on the potential function $\eta$ and it is natural to ask if there is some example of such a potential function where eigenvalue estimates do not depend on $C_0$. In the second part of our Corollary~\ref{gaussian-corollary}, we give a positive answer to this question, that is, we give an example where the estimates do not depend on the constant $C_0$.

When $\dv{T} = 0$ the $(1,1)$-tensor $T$ is called {\it divergence-free} or, for reasons of physical conservation laws, is called a {\it locally conserved} tensor (Cf. Gover and Orsted~\cite{GoverOrsted}). Divergence-free tensors are important in physical facts and appear naturally in fluid dynamics, for instance, in the study of compressible gas; rarefied gas; steady/self-similar flows, and relativistic gas dynamics, for more details see Serre~\cite{Serre}.
\begin{example}
Let $f$ be a smooth function on a Riemannian manifold $(M, \langle , \rangle)$ and define
\begin{equation*}
    T_f:= df \otimes df - \frac{|\nabla f |^2}{2}\langle , \rangle.
\end{equation*}
When $\Delta f =0$, we can see that the symmetric tensor $T_f$ is divergence-free.
\end{example}
\begin{example}
 Let $(M^n, \langle , \rangle)$ be an $n (\geq 3)$-dimensional Einstein manifold, that is, $\rc=\rho \langle , \rangle$ for some constant $\rho$. If $\rho >0$, then $\rc$ is a tensor symmetric and positive definite which is divergence-free. Moreover, if $\rho<0$ the Einstein tensor 
 \begin{equation*}
    E = \rc - \frac{R_0}{2}\langle , \rangle,
\end{equation*}
where $\rc$ is the Ricci tensor and $R_0=\tr(\rc)$, is a tensor symmetric and positive definite which is divergence-free.
\end{example}

Cheng and Yau~\cite{ChengYau} introduced a differential operator appropriate for the study of complete hypersurfaces of constant scalar curvature in space forms, namely
\begin{equation*}
    \square f = \tr{(\nabla^2f \circ T)}=\langle \nabla^2 f, T\rangle,
\end{equation*}
where $f\in C^\infty(M)$ and $T$ is a symmetric $(1, 1)$-tensor.
In fact, with a careful study of this operator, using a divergence-free tensor, Cheng and Yau obtained remarkable rigidity results for such hypersurfaces. It is worth mentioning the relationship, observed by Gomes and Miranda~\cite{GomesMiranda}, between operator $\mathscr{L}$ and operator $\square$, that is, we can see that operator $\mathscr{L}$ is a first-order perturbation of the  Cheng-Yau's operator. In fact, when $T$ is divergence-free, from \cite[Eq.~2.3]{GomesMiranda} the $\mathscr{L}$ operator becomes
\begin{equation*}
    \mathscr{L}f = \square f - \langle\nabla \eta, T(\nabla f)\rangle.
\end{equation*}
Therefore, it is a {\it drifted Cheng-Yau operator} with drifting function $\eta$. Hence, when $T$ is divergence-free, let us call the $\mathscr{L}^2$ operator the {\it bi-drifted Cheng-Yau operator}.

Now let us apply our results to the {\it drifted Cheng-Yau operator}. For this, notice that if $T$ is divergence-free, then $\tr{(\nabla T)=0}$ and so $T_0=0$ (cf. Section~\ref{preliminaries}), which, combining with Theorem~\ref{theorem1.1} implies the following. 
\begin{cor}\label{cor-T-divergence-free} 
Let $\Omega$ be a bounded domain in an $n$-dimensional complete Riemannian manifold $M^n$ isometrically immersed in $\mathbb{R}^m$. Denote by $\Gamma_i$ the $i$-th eigenvalue  of the Dirichlet problem for the bi-drifted Cheng-Yau operator, then we have
\begin{align}\label{Equation-1.4}
    \sum_{i=1}^k(\Gamma_{k+1}-\Gamma_i)^2 \leq& \frac{1}{n\varepsilon}\Bigg\{\sum_{i=1}^k(\Gamma_{k+1}-\Gamma_i)^2\Big[(4+2n)\delta\Gamma_i^{\frac{1}{2}} + n^2H_0^2+4C_0 \Big]\Bigg\}^{\frac{1}{2}}\nonumber\\
    &\times\Bigg\{\sum_{i=1}^k(\Gamma_{k+1}-\Gamma_i)\Big(4\delta \Gamma_i^{\frac{1}{2}}+ n^2H_0^2+4C_0 \Big)\Bigg\}^{\frac{1}{2}},
\end{align}
and
\begin{align}\label{Equation-1.14}
    \sum_{i=1}^k(\Gamma_{k+1}-\Gamma_1)^\frac{1}{2}\leq& \frac{1}{\varepsilon}\Bigg[\Big(4\delta \Gamma_1^{\frac{1}{2}} + n^2H_0^2+4C_0\Big)\Big((4 + 2n)\delta \Gamma_1^{\frac{1}{2}} + n^2H_0^2+4C_0\Big)\Bigg]^{\frac{1}{2}},
\end{align}
where $C_0$ is given by \eqref{C_0}.
\end{cor}

Using an algebraic lemma obtained by Jost et al.~\cite[Lemma~2.3]{Jost} and \eqref{Equation-1.4}, we get the next result.
\begin{cor}\label{cor-T-divergence-free2} 
Under the same setup as in Corollary~\ref{cor-T-divergence-free}, we have
\begin{align}\label{Equation-1.6}
    &\sum_{i=1}^k(\Gamma_{k+1}-\Gamma_i)^2\nonumber \\
    &\leq\frac{1}{n^2\varepsilon^2}\sum_{i=1}^k(\Gamma_{k+1}-\Gamma_i)\Big[(4+2n\delta)\Gamma_i^{\frac{1}{2}} + n^2H_0^2+4C_0 \Big]\Big(4\delta \Gamma_i^{\frac{1}{2}}+ n^2H_0^2+4C_0 \Big),
\end{align}
\end{cor}

Since Inequality~\eqref{Equation-1.6} is a quadratic inequality of $\Gamma_{k+1}$, solving it we can to obtain an upper bound on $\Gamma_{k+1}$ in terms of the first $k$ eigenvalues and $H_0^2$.
\begin{cor}\label{cor-T-divergence-free3} 
Under the same setup as in Corollary~\ref{cor-T-divergence-free}, we have 
\begin{align}\label{Equation-1.7}
    \Gamma_{k+1}\leq A_k + \sqrt{A_k^2-B_k},
\end{align}
in particular, we have
\begin{align}\label{Equation-1.8}
     \Gamma_{k+1} - \Gamma_k \leq 2\sqrt{A_k^2-B_k},
\end{align}
where 
\begin{align*}
    A_k = \frac{1}{k}\Bigg\{\sum_{i=1}^k\Gamma_i + \frac{1}{2n^2\varepsilon^2}\sum_{i=1}^k\Big[(4+2n)\delta\Gamma_i^{\frac{1}{2}} + n^2H_0^2+4C_0 \Big]\Big(4\delta \Gamma_i^{\frac{1}{2}}+ n^2H_0^2+4C_0 \Big)\Bigg\},
\end{align*}
and
\begin{align*}
    B_k = \frac{1}{k}\Bigg\{\sum_{i=1}^k\Gamma_i^2 + \frac{1}{n^2\varepsilon^2}\sum_{i=1}^k\Gamma_i\Big[(4+2n)\delta\Gamma_i^{\frac{1}{2}} + n^2H_0^2+4C_0 \Big]\Big(4\delta \Gamma_i^{\frac{1}{2}}+ n^2H_0^2+4C_0 \Big)\Bigg\}.
\end{align*}
\end{cor}
\begin{remark}
It should be mentioned that \eqref{Equation-1.6} generalizes Inequality~\eqref{Cheng-etal-estimate} and Corollary~\ref{cor-T-divergence-free3} generalizes \cite[Corolaries~1 and 2]{Cheng-etal} both obtained by Cheng et al. in \cite{Cheng-etal}. This can be seen by taking $T = I$ and $\eta = constant$ in our result. In addition, in a sense, our Corollary~\eqref{cor-T-divergence-free} generalizes some results obtained by Du et al.~\cite{Du-etal} for the drifted Laplacian, see concluding remarks in Section~\ref{concludingremarks}.
\end{remark}

Now, let us apply some of our results to Gaussian shrinking soliton. For this, remember that the Bakry-Emery Ricci curvature on $M$ is  defined by
\begin{equation*}
    \rc_\eta = \rc + \nabla^2 \eta,
\end{equation*}
where $\rc$ is the Ricci curvature on $M$ and $\nabla^2$ is the Hessian operator. Therefore, the triple $(M, \langle , \rangle, \eta)$ is called a gradient Ricci soliton when the Bakry-Emery Ricci curvature satisfies the equation $\rc_\eta = \lambda\langle , \rangle$, for some constant $\lambda$. The gradient Ricci soliton is classified according to the  sign of $\lambda$, that is, it is called steady for $\lambda = 0$, shrinking for $\lambda > 0$ and expanding for $\lambda<0$.
\begin{example}
The Gaussian shrinking soliton is $(\mathbb{R}^m, \langle , \rangle_{can},\frac{1}{4}|x|^2)$, where $\langle , \rangle_{can}$ is the standard Euclidean metric on $\mathbb{R}^m$, $\eta = \frac{1}{4}|x|^2$ and $x\in\mathbb{R}^m$. In this case, since $\eta = \frac{1}{4}|x|^2$ we have $\Delta \eta = \frac{m}{2}$ and $|\nabla \eta|^2=\frac{1}{4}|x|^2$ hence we can get
\begin{equation*}
    C_0 = \sup_\Omega \Big\{\frac{1}{2}\Delta \eta - \frac{1}{4}|\nabla \eta|^2\Big\} \leq \frac{m}{4}-\frac{1}{16}\min_{\bar{\Omega}} \{|x|^2\},
\end{equation*}
in the case where $T=I$.
\end{example}

From the previous inequality and Corollary~\ref{cor-T-divergence-free} we can to obtain the following eigenvalue estimate of Problem~\eqref{problem3} on the Gaussian shrinking soliton. 
\begin{cor}\label{gaussian-corollary}
Let $\Omega$ be a connected bounded domain in the Gaussian shrinking soliton $(\mathbb{R}^n, \langle , \rangle_{can}, \frac{1}{4}|x|^2)$, and let $\lambda_i$ be the i-th eigenvalue of Problem~\ref{problem3}, that is,
\begin{equation*}
    \left\{\begin{array}{ccccc} 
    \Delta_\eta^2  u  &=&  \Gamma u & \mbox{in } & \Omega,\\
     u=\frac{\partial u}{\partial \nu}&=&0 & \mbox{on} & \partial\Omega. 
    \end{array} 
    \right.
\end{equation*}
then we have
\begin{align*}
    \sum_{i=1}^k(\Gamma_{k+1}-\Gamma_i)^2 \leq& \frac{1}{n}\Bigg\{\sum_{i=1}^k(\Gamma_{k+1}-\Gamma_i)^2\Big[(4+2n)\Gamma_i^{\frac{1}{2}} +\Big(m-\frac{1}{4}\min_{\bar{\Omega}} \{|x|^2\}\Big)\Big]\Bigg\}^{\frac{1}{2}}\nonumber\\
    &\times\Bigg\{\sum_{i=1}^k(\Gamma_{k+1}-\Gamma_i)\Big[4 \Gamma_i^{\frac{1}{2}}+\Big(m-\frac{1}{4}\min_{\bar{\Omega}} \{|x|^2\} \Big)\Big]\Bigg\}^{\frac{1}{2}},
\end{align*}
and
\begin{align*}
    \sum_{i=1}^k&(\Gamma_{i+1}-\Gamma_1)^\frac{1}{2}\\
    &\leq \Bigg\{\Big[4 \Gamma_1^{\frac{1}{2}} + \Big(m-\frac{1}{4}\min_{\bar{\Omega}} \{|x|^2\}\Big)\Big]\Big[(4 + 2n) \Gamma_1^{\frac{1}{2}} + \Big(m-\frac{1}{4}\min_{\bar{\Omega}} \{|x|^2\}\Big)\Big]\Bigg\}^{\frac{1}{2}}.
\end{align*}
In particular, if $\Omega = \Big\{ x \in \mathbb{R}^n; 4m < |x|^2 < r_0^2 \Big\}$ for any real constant $r_0^2 > 4m$, then we have
\begin{align}\label{Equation-(1.17)}
    \sum_{i=1}^k(\Gamma_{k+1}-\Gamma_i)^2 \leq& \frac{1}{n}\Bigg\{\sum_{i=1}^k(\Gamma_{k+1}-\Gamma_i)^2(4+2n)\Gamma_i^{\frac{1}{2}}\Bigg\}^{\frac{1}{2}}\Bigg\{\sum_{i=1}^k(\Gamma_{k+1}-\Gamma_i)4 \Gamma_i^{\frac{1}{2}}\Bigg\}^{\frac{1}{2}},
\end{align}
and
\begin{align}\label{Equation-(1.18)}
    \sum_{i=1}^k&(\Gamma_{k+1}-\Gamma_1)^\frac{1}{2}\leq \Big\{8(2 + n)\Gamma_1\Big\}^{\frac{1}{2}}.
\end{align}
\end{cor}
\begin{remark}
The first part of our corollary is \cite[Theorem~1.2]{Du-etal} obtained by Du et al. for Problem~\eqref{problem3}, on the Gaussian shrinking soliton. 
\end{remark}
\begin{remark}
Is worth mentioning that the inequalities \eqref{Equation-(1.17)} and \eqref{Equation-(1.18)} are inequalities for Problem~\ref{problem3} which do not depend on the constant $C_0$ and it has the same behavior as estimates obtained for Problem~\ref{problem2} for complete minimal submanifold in a Euclidean space, for seeing this the reader can check Inequality~\eqref{WangXia-estimate}(with $H_0=0$) and \cite[Corollary~2]{Cheng-etal1}, respectively.
\end{remark}

To end this section, we present eigenvalue estimates for Problem~\eqref{problem1} on Riemannian manifolds that admit some special function. This type of result was initially obtained in the work by do Carmo et al.~\cite{doCarmoWangXia}.
\begin{theorem}\label{theorem3}
Let $M$ be an $n$-dimensional complete Riemannian manifold and let $\Omega$ be a bounded domain with smooth boundary in $M$.  We denote by $\Gamma_i$ the $i$-th eigenvalue of Problem~\eqref{problem1} and assume that $\mathscr{L}f \leq \delta \Delta_\eta f$, for all $f\in C^2(\Omega)$.
\begin{enumerate}
    \item[i)] If there exists a function $\varphi: \Omega \to \mathbb{R}$ such that
\begin{equation}\label{specialfunction-1}
    |\nabla \varphi| = 1, \quad |\Delta\varphi| \leq A_0, \quad \mbox{on} \quad \Omega,
\end{equation}
then, we have
\begin{align}\label{thm3-i}
    \sum_{i=1}^k(\Gamma_{k+1}-\Gamma_i)^2 \leq& \frac{\delta}{\varepsilon}\Bigg\{\sum_{i=1}^k(\Gamma_{k+1}-\Gamma_i)^2\Big(6\Gamma_i^{\frac{1}{2}}+ 4\delta^{\frac{1}{2}}(A_0 + \eta_0)\Gamma_i^\frac{1}{4} + \delta(A_0+\eta_0)^2\Big)\Bigg\}^\frac{1}{2}\nonumber\\
   & \times \Bigg\{\sum_{i=1}^k(\Gamma_{k+1} - \Gamma_i)\Big(4\Gamma_i^{\frac{1}{2}} + 4\delta^{\frac{1}{2}}(A_0 + \eta_0)\Gamma_i^{\frac{1}{4}} + \delta(A_0 + \eta_0)^2 \Big) \Bigg\}^\frac{1}{2};
\end{align}
\item[ii)] If $M$ admits an eigenmap $f=(f_1, f_2, \ldots, f_{m+1}): \Omega \to \mathbb{S}^m$ corresponding to an eigenvalue $\gamma$, that is,
\begin{equation}\label{specialfunction-2}
    \sum_{\alpha=1}^{m+1}f_\alpha^2 = 1, \quad \Delta f_\alpha = - \gamma f_\alpha, \quad \alpha = 1, \ldots, m+1,
\end{equation}
then, we have
\begin{align}\label{thm3-ii}
    \sum_{i=1}^k(\Gamma_{k+1}-\Gamma_i)^2 \leq& \frac{\delta}{\varepsilon}\Bigg\{\sum_{i=1}^k(\Gamma_{k+1}-\Gamma_i)^2\Big(6\Gamma_i^{\frac{1}{2}} + 4\delta^{\frac{1}{2}}\eta_0\Gamma_i^{\frac{1}{4}} + \delta^2(\gamma + \eta_0)\Big)\Bigg\}^\frac{1}{2}\nonumber\\
   & \times \Bigg\{\sum_{i=1}^k(\Gamma_{k+1} - \Gamma_i)\Big(4\Gamma_i^{\frac{1}{2}} + 4\delta^{\frac{1}{2}}\eta_0\Gamma_i^{\frac{1}{4}} + \delta^2(\gamma + \eta_0)\Big) \Bigg\}^\frac{1}{2};
\end{align}
\end{enumerate}
In the above, $A_0$ is constant and $\mathbb{S}^m$ is the unit $m$-sphere.
\end{theorem}

Examples of special functions satisfying the conditions in Theorem~\ref{theorem3} can be found in \cite[Exemples~4.1-4.4]{doCarmoWangXia} or \cite[Exemples~2.1-2.4]{WangXia}.

\section{Preliminaries}\label{preliminaries}
This section is short and serves to establish some basic notations and describe what is meant by properties of a $(1,1)$-tensor in a bounded domain $\Omega\subset M^n$ with smooth boundary $\partial\Omega$.

Let us identify, through this paper, a $(0,2)$-tensor $T:\mathfrak{X}(\Omega)\times\mathfrak{X}(\Omega)\to C^{\infty}(\Omega)$ with its associated $(1,1)$-tensor $T:\mathfrak{X}(\Omega)\to\mathfrak{X}(\Omega)$ by the equation 
\begin{equation*}
    \langle T(X), Y \rangle = T(X, Y).
\end{equation*}
In particular, let us identify the metric tensor $\langle , \rangle$ with the identity $I$ in $\mathfrak{X}(\Omega)$. We can see that $\varepsilon I \leq T \leq \delta I$ on $\Omega$ implies 
\begin{align*}
    \varepsilon \langle T(X), X\rangle \leq |T(X)|^2 \leq \delta \langle T(X), X\rangle, \quad \mbox{for all} \quad X\in \mathfrak{X}(\Omega).
\end{align*} And from there, we get
\begin{align}\label{Equation-2.1}
    \varepsilon^2 |\nabla \eta|^2 \leq |T(\nabla \eta)|^2 \leq \delta^2 |\nabla \eta|^2.
\end{align}
We will use the following notations
\begin{align*}
    \tr(\nabla T):=\sum_{i=1}^n(\nabla T)(e_i, e_i), \quad |T|=\Big(\sum_{i=1}^n|T(e_i)|^2\Big)^{\frac{1}{2}},
\end{align*}
and $\|\cdot\|_{L^2}$ for the canonical norm of a real-valued function in $L^2(\Omega,dm)$.

Let $\alpha$ be the second fundamental form on an $n$-dimensional complete Riemannian manifold $(M^n, \langle , \rangle)$  isometrically immersed in $\mathbb{R}^m$, so ${\bf H}=\frac{1}{n}\tr (\alpha)$ is the mean curvature vector. The generalized mean curvature vector, associate with a symmetric $(1, 1)$-tensor $T$, is the normal vector field defined by
\begin{equation*}
    {\bf H}_T=\frac{1}{n}\sum_{i,j=1}^nT(e_i, e_j)\alpha(e_i, e_j)=\frac{1}{n}\sum_{i=1}^n\alpha(T(e_i), e_i):=\frac{1}{n}\tr{(\alpha\circ T)},
\end{equation*}
where $\{e_1, e_2, \ldots, e_n\}$ is a local orthonormal frame of $TM$. Notice that, when $T=I$ we have ${\bf H}_T={\bf H}$. 

From the definition of $\eta$-divergence of $X$ (see Eq.~\eqref{eq1.1}) and the usual properties of divergence of vector fields, one has
\begin{equation*}
    \dv_\eta(fX)=f\dv_\eta X + \langle \nabla f, X \rangle,
\end{equation*}
for all $f \in C^\infty(\Omega)$. Notice that the $(\eta,T)$-divergence form of $\mathscr{L}$ on $\Omega$ allows us to check that the divergence theorem remains true in the form
\begin{equation*}
   \int_\Omega\dv_\eta X dm = \int_{\partial\Omega} \langle X, \nu\rangle d\mu.
\end{equation*}
In particular, for $X=T(\nabla f)$, since $T(\nabla f, \nu) = \langle T(\nabla f), \nu \rangle = \frac{\partial f}{\partial \nu_T}$ we obtain
\begin{equation*}
   \int_\Omega\mathscr{L}{f}dm = \int_{\partial\Omega}T(\nabla f, \nu) d\mu = \int_{\partial \Omega} \frac{\partial f}{\partial \nu_T}d\mu, 
\end{equation*}
where $dm = e^{-\eta}d\Omega$ and $d\mu = e^{-\eta}d\partial\Omega$ are the weight volume form on $\Omega$ and the volume form on the boundary $\partial\Omega$ induced by the outward unit normal vector $\nu$ on $\partial \Omega$, respectively. Thus, the integration by parts formula is given by
\begin{equation}\label{parts}
     \int_{\Omega}\ell\mathscr{L}{f}dm =-\int_\Omega T(\nabla\ell, \nabla f)dm + \int_{\partial\Omega}\ell \frac{\partial f}{\partial \nu_T} d\mu,
\end{equation}
for all $\ell, f \in C^\infty(\Omega)$. 

Therefore, we can see that $\mathscr{L}^2$ is a formally self-adjoint operator in the space of all smooth real-valued functions $u, v$ such that $u|_{\partial \Omega}= v|_{\partial \Omega}=\frac{\partial u}{\partial \nu_T}|_{\partial \Omega}= \frac{\partial v}{\partial \nu_T}|_{\partial \Omega}=0$, with respect to the inner product
\begin{align*}
    \langle\langle u , v \rangle \rangle = \int_\Omega uv dm,
\end{align*}
In fact, from \eqref{parts} and since $u|_{\partial \Omega}= v|_{\partial \Omega}=0$, we have
\begin{align}\label{Equation-2.2}
    \int_{\Omega} u \mathscr{L}{v}dm =-\int_\Omega T(\nabla u, \nabla v)dm = \int_{\Omega} v \mathscr{L}{u}dm.
\end{align}
Now, since $\frac{\partial u}{\partial \nu_T}|_{\partial \Omega}= \frac{\partial v}{\partial \nu_T}|_{\partial \Omega}=0$, again from \eqref{parts} we get
\begin{align*}
    0=\int_{\partial \Omega} \mathscr{L}u \frac{\partial v}{\partial \nu_T}d\mu = \int_\Omega \mathscr{L}u \mathscr{L}v dm + \int_\Omega T(\nabla \mathscr{L}u, \nabla v)dm,
\end{align*}
replacing $u$ by $v$ in the previous equality we obtain
\begin{align}\label{Equation-2.3}
    - \int_\Omega T(\nabla \mathscr{L}u, \nabla v)dm = \int_\Omega \mathscr{L}u \mathscr{L}v dm = - \int_\Omega T(\nabla u, \nabla \mathscr{L}v)dm.
\end{align}
Therefore, from \eqref{Equation-2.2} and \eqref{Equation-2.3} we have
\begin{align*}
    \langle\langle u, \mathscr{L}^2v\rangle\rangle &= \int_\Omega u \mathscr{L}^2v dm = \int_\Omega u \mathscr{L}(\mathscr{L} v)dm = - \int_\Omega T(\nabla u, \nabla \mathscr{L} v)dm \\
    &= - \int_\Omega T(\nabla \mathscr{L} u, \nabla  v)dm = \int_\Omega v \mathscr{L}(\mathscr{L} u)dm = \int_\Omega  v\mathscr{L}^2u dm \\
    & = \langle\langle \mathscr{L}^2u, v\rangle\rangle. 
\end{align*}
 Thus the eigenvalue problem~\eqref{problem1} has a real and discrete spectrum 
\begin{equation*}
    0 < \Gamma_1 \leq \Gamma_2 \leq \cdots \leq \Gamma_k \leq \cdots\to\infty,
\end{equation*}
where each $\Gamma_i$ is repeated according to its multiplicity. In particular, for eigenfunction $u_i$ corresponding to the eigenvalue $\Gamma_i$, from Problem~\eqref{problem1} and integration by parts, we have
\begin{equation}\label{lambda_i}
    \Gamma_i\int_\Omega u_i^2dm = \int_\Omega u_i\mathscr{L}^2u_idm =  \int_\Omega (\mathscr{L}u_i)^2dm.
\end{equation}

It is worth mentioning the work by Gomes and Miranda~\cite[Section~2]{GomesMiranda} where they gave geometric motivations to work with the operator $\mathscr{L}$ in the $(\eta,T)$-divergence form in bounded domains in Riemannian manifolds, in addition to calculating a Bochner-type formula for it. 

\section{Two keystone technical lemmas} 

In order to prove our first theorem we will need of the two keystone technical lemmas. 
\begin{lem}\label{lemma1}
Let $\Omega$ be a bounded domain in an $n$-dimensional complete Riemannian manifold $M$. Let $\Gamma_i$ be the i-th eigenvalue of Problem~\eqref{problem1} and let $u_i$ be a normalized real-valued eigenfunction corresponding to $\Gamma_i$, that is,
\begin{align*}
    \mathscr{L}^2u_i = \Gamma_iu_i \quad \mbox{in} \quad \Omega, \quad u_i|_{\partial \Omega}=\frac{\partial u_i}{\partial \nu_T}|_{\partial \Omega} = 0, \quad \int_\Omega u_iu_jdm = \delta_{ij}, \quad \forall i,j= 1, 2, \ldots,
\end{align*}
where $\nu$ is an outward normal vector field of $\partial \Omega$ and $\frac{\partial u}{\partial \nu_T}=\langle T(\nabla u), \nu \rangle$.
Then,
\begin{enumerate}
    \item[i)] for any smooth function $f: \Omega \to \mathbb{R}$, we have \begin{align}\label{lemma-i)-estimate}
    \sum_{i=1}^k&(\Gamma_{k+1}-\Gamma_i)^2 \int_{\Omega}T(\nabla f, \nabla f)u_i^2dm\leq B\sum_{i=1}^k(\Gamma_{k+1}-\Gamma_i)^2  \nonumber\\
    &\times\int_\Omega \Big[u_i^2(\mathscr{L}f)^2 + 4\Big((T(\nabla f, \nabla u_i))^2 + u_i\mathscr{L}f T(\nabla f, \nabla u_i)\Big)-2u_i\mathscr{L}u_iT(\nabla f, \nabla f)\Big]dm\nonumber\\
    &+\frac{1}{B}\sum_{i=1}^k(\Gamma_{k+1}-\Gamma_i)\Big\|T(\nabla f, \nabla u_i) + \frac{u_i}{2}\mathscr{L}f \Big\|^2_{L^2};
\end{align}
   \item[ii)] if $h_i: \Omega \to \mathbb{R}$ is a smooth function that satisfies $\int_\Omega h_i u_1 u_{j+1} dm =0$ for $1 \leq j < i$, for any positive integer $i\geq 2$, we have
   \begin{align}\label{lemma-ii)-estimate}
       (\Gamma&_{i+1}-\Gamma_1)^{\frac{1}{2}}\int_\Omega u_1^2 T(\nabla h_i, \nabla h_i) dm  \nonumber\\ &\leq \Big(\frac{B}{2} + \frac{1}{2B}\Big)\int_\Omega(u_1 \mathscr{L}h_i + 2T(\nabla h_i, \nabla u_1))^2dm - B \int_\Omega u_1 \mathscr{L}u_1 T(\nabla h_i, \nabla h_i)dm,
   \end{align}
\end{enumerate}
where $B$ is any positive constant.
\end{lem}
\begin{proof}
For $i=1, \ldots, k$, consider the following functions $\phi_i: \Omega \to \mathbb{R}$ given by
\begin{align*}
    \phi_i = fu_i - \sum_{j=1}^k a_{ij}u_j,
\end{align*}
where $a_{ij}=\int_\Omega fu_iu_j dm$. Notice that $\phi_i|_{\partial \Omega}=\frac{\partial \phi_i}{\partial \nu_T}|_{\partial \Omega}=0$ and
\begin{align*}
    \int_\Omega u_j \phi_i dm =0, \quad \mbox{for all} \quad i = 1, \ldots, k.
\end{align*}
Then from the Rayleigh-Ritz inequality (see, e.g., \cite[Theorem~9.43]{PJ-Olver}), we get
\begin{equation}\label{Rayleigh}
    \Gamma_{k+1} \leq - \frac{\int_\Omega \phi_i\mathscr{L}^2\phi_idm}{\int_\Omega \phi_i^2dm}, \quad \mbox{for all} \quad i = 1, \ldots, k.
\end{equation}
Since
\begin{align*}
    \mathscr{L}\phi_i = \mathscr{L}(fu_i)-\sum_{j=1}^ka_{ij}\mathscr{L}u_j = f \mathscr{L}u_i + u_i\mathscr{L}f + 2T(\nabla f, \nabla u_i) - \sum_{j=1}^ka_{ij}\mathscr{L}u_j,
\end{align*}
we have
\begin{align*}
    \mathscr{L}^2\phi_i = \mathscr{L}f\mathscr{L}u_i + \Gamma_i fu_i +2T(\nabla f, \nabla \mathscr{L}u_i)+\mathscr{L}(u_i\mathscr{L}f)+2\mathscr{L}T(\nabla f, \nabla u_i) - \sum_{j=1}^k a_{ij}\Gamma_ju_j,
\end{align*}
hence, we obtain
\begin{align}\label{Equation-3.4}
    \int_\Omega \phi_i\mathscr{L}^2\phi_idm = \Gamma_i\|\phi_i\|_{L^2}^2 + \int_\Omega fu_ip_idm - \sum_{j=1}^k a_{ij}r_{ij},
\end{align}
where $r_{ij}=\int_\Omega p_iu_j dm$ and 
\begin{align*}
    p_i = \mathscr{L}f\mathscr{L}u_i +2T(\nabla f, \nabla \mathscr{L}u_i)+\mathscr{L}(u_i\mathscr{L}f)+2\mathscr{L}T(\nabla f, \nabla u_i). 
\end{align*}
Using the divergence theorem, it is not difficult to see that 
\begin{align}\label{Equation-3.5}
    \int_\Omega & u_j \mathscr{L}T(\nabla f, \nabla u_i)dm + \int_\Omega u_j T(\nabla f, \nabla \mathscr{L} u_i) dm \nonumber\\ 
    &= \int_\Omega \mathscr{L}u_j T(\nabla f, \nabla u_i)dm
    -\int_\Omega \mathscr{L}u_i T(\nabla f, \nabla u_j)dm - \int_\Omega u_j \mathscr{L}f\mathscr{L}u_i dm. 
\end{align}
Moreover, from \eqref{parts} and \eqref{lambda_i}
\begin{align*}
    \int_\Omega& \mathscr{L}u_j T(\nabla f, \nabla u_i)dm
    -\int_\Omega \mathscr{L}u_i T(\nabla f, \nabla u_j)dm \\
    =& -\int_\Omega fT(\nabla \mathscr{L}u_j, \nabla u_i) dm + \int_\Omega f T(\nabla  \mathscr{L}u_i, \nabla u_j) dm \\
    =&\int_\Omega (fu_i  \mathscr{L}^2u_j - fu_j  \mathscr{L}^2u_i)dm +\int_\Omega u_i T(\nabla f, \nabla  \mathscr{L}u_j) dm - \int_\Omega u_j T(\nabla f, \nabla  \mathscr{L}u_i)dm \\
    =&(\Gamma_j - \Gamma_i)a_{ij} - \int_\Omega u_i  \mathscr{L}u_j  \mathscr{L}f dm -\int_\Omega \mathscr{L}u_jT(\nabla f, \nabla u_i)dm \\
    &+ \int_\Omega u_j  \mathscr{L}u_i  \mathscr{L}f dm + \int_\Omega \mathscr{L}u_iT(\nabla f, \nabla u_j)dm,
\end{align*}
and the previous equality implies that
\begin{align}\label{Equation-3.6}
     2&\int_\Omega \mathscr{L}u_j T(\nabla f, \nabla u_i)dm
    -2\int_\Omega \mathscr{L}u_i T(\nabla f, \nabla u_j)dm = \nonumber\\
    &=(\Gamma_j - \Gamma_i)a_{ij} - \int_\Omega u_i  \mathscr{L}u_j  \mathscr{L}f dm + \int_\Omega u_j  \mathscr{L}u_i  \mathscr{L}f dm. 
\end{align}
From \eqref{Equation-3.5} and \eqref{Equation-3.6}, we obtain
\begin{align}\label{Equation-3.7}
    2\int_\Omega & u_j \mathscr{L}T(\nabla f, \nabla u_i)dm + 2\int_\Omega u_j T(\nabla f, \nabla \mathscr{L} u_i) dm \nonumber\\
    &=(\Gamma_j - \Gamma_i)a_{ij} - \int_\Omega u_i  \mathscr{L}u_j  \mathscr{L}f dm - \int_\Omega u_j  \mathscr{L}u_i  \mathscr{L}f dm. 
\end{align}
Since $\mathscr{L}$ is self-adjoint, we know that
\begin{align*}
    \int_\Omega u_j  \mathscr{L}(u_i  \mathscr{L}f)dm = \int_\Omega u_i  \mathscr{L}u_j  \mathscr{L}f dm.
\end{align*}
From \eqref{Equation-3.7} and the previous equality, we have
\begin{align*}
    \int_\Omega u_j \Big(\mathscr{L}u_i \mathscr{L}f + \mathscr{L}(u_i\mathscr{L}f) + 2 \mathscr{L}T(\nabla f, \nabla \mathscr{L}u_i) + 2T(\nabla f, \nabla \mathscr{L}u_i) \Big)dm = (\Gamma_j - \Gamma_i)a_{ij},
\end{align*}
since $r_{ij}=\int_\Omega p_iu_j dm$, indeed we obtain
\begin{align}\label{Equation-3.8}
    r_{ij} = (\Gamma_j - \Gamma_i)a_{ij}.
\end{align}
From \eqref{Rayleigh}, \eqref{Equation-3.4} and \eqref{Equation-3.8} it follows that
\begin{align}\label{Equation-3.9}
    (\Gamma_{k+1}-\Gamma_i)\|\phi_i\|_{L^2}^2 \leq \int_\Omega fu_ip_i dm - \sum_{j=1}^k a_{ij}r_{ij} = w_i + \sum_{j=1}^k (\Gamma_i - \Gamma_j)a_{ij}^2,
\end{align}
where $w_i = \int_\Omega fu_ip_i dm$. Now, setting
\begin{align*}
    b_{ij}=\int_\Omega u_j\Big(T(\nabla f, \nabla u_i)+\frac{u_i}{2}\mathscr{L}f\Big)dm,
\end{align*}
then, from integration by parts formula \eqref{parts} we have
\begin{align*}
    b_{ij}+b_{ji}&=\int_\Omega T(\nabla f, u_j\nabla u_i+u_i \nabla u_j)dm +\int_\Omega u_i u_j\mathscr{L}fdm\\
    &=\int_\Omega T(\nabla f, \nabla (u_i u_j))dm +\int_\Omega u_i u_j\mathscr{L}fdm\\
    &=-\int_\Omega u_i u_j\mathscr{L}fdm +\int_\Omega u_i u_j\mathscr{L}fdm=0.
\end{align*}
Now, one gets from the divergence theorem and definition of $\phi_i$ that
\begin{align}\label{Equation-3.10}
    \int_\Omega (-2)\phi_i\Big(T(\nabla f, \nabla u_i)+\frac{u_i}{2}\mathscr{L}f\Big)dm &= -2\int_\Omega fu_i\Big(T(\nabla f, \nabla u_i)+\frac{u_i}{2}\mathscr{L}f\Big)dm +2\sum_{j=1}^ka_{ij}b_{ij}\nonumber\\
    &=\int_\Omega u_i^2T(\nabla f, \nabla f) dm + 2\sum_{j=1}^ka_{ij}b_{ij}.
\end{align}
Multiplying \eqref{Equation-3.10} by $(\Gamma_{k+1} - \Gamma_i)^2$ and using the Schwarz inequality and \eqref{Equation-3.9}, for any positive constant $B$, we get
\begin{align*}
   (\Gamma&_{k+1} - \Gamma_i)^2\Bigg( \int_\Omega u_i^2T(\nabla f, \nabla f) dm + 2\sum_{j=1}^ka_{ij}b_{ij}\Bigg)\\
   =&(\Gamma_{k+1} - \Gamma_i)^2\int_\Omega (-2)\phi_i\Big(T(\nabla f, \nabla u_i)+\frac{u_i}{2}\mathscr{L}f-\sum_{j=1}^kb_{ij}u_j\Big)dm\\
   \leq& 2(\Gamma_{k+1} - \Gamma_i)^2\|\phi_i\|_{L^2}\Big\|T(\nabla f, \nabla u_i ) +\frac{u_i}{2}\mathscr{L}f-\sum_{j=1}^kb_{ij}u_j\Big\|_{L^2}\\
   \leq& B(\Gamma_{k+1} - \Gamma_i)^3\|\phi_i\|_{L^2}^2 + \frac{(\Gamma_{k+1} - \Gamma_i)}{B}\Big\|T(\nabla f, \nabla u_i ) +\frac{u_i}{2}\mathscr{L}f-\sum_{j=1}^kb_{ij}u_j\Big\|_{L^2}^2\\
   \leq&  B(\Gamma_{k+1} - \Gamma_i)^2(w_i + \sum_{j=1}^k (\Gamma_i - \Gamma_j)a_{ij}^2)\\
   &+ \frac{(\Gamma_{k+1} - \Gamma_i)}{B}\Bigg(\Big\|T(\nabla f, \nabla u_i ) +\frac{u_i}{2}\mathscr{L}f\Big\|_{L^2}^2 -\sum_{j=1}^kb_{ij}^2\Bigg).
\end{align*}
Summing over $i$ from $1$ to $k$ and noticing $a_{ij}=a_{ji}$ and $b_{ij}=-b_{ji}$, we conclude that
\begin{align*}
   \sum_{i=1}^k&(\Gamma_{k+1} - \Gamma_i)^2\int_\Omega u_i^2T(\nabla f, \nabla f) dm -2\sum_{i,j=1}^k(\Gamma_{k+1} - \Gamma_i)(\Gamma_i - \Gamma_j)a_{ij}b_{ij}\\
   \leq& B\sum_{i=1}^k(\Gamma_{k+1} - \Gamma_i)^2w_i -B \sum_{i,j=1}^k(\Gamma_{k+1} - \Gamma_i)(\Gamma_i - \Gamma_j)^2a_{ij}^2 - \frac{1}{B}\sum_{i,j=1}^k(\Gamma_{k+1} - \Gamma_i)b_{ij}^2\\
   &+\frac{1}{B}\sum_{i=1}^k(\Gamma_{k+1} - \Gamma_i)\Big\|T(\nabla f, \nabla u_i ) +\frac{u_i}{2}\mathscr{L}f\Big\|_{L^2}^2.
\end{align*}
Hence, since $B>0$, we obtain
\begin{align}\label{Equation-3.11}
   \sum_{i=1}^k&(\Gamma_{k+1} - \Gamma_i)^2\int_\Omega u_i^2T(\nabla f, \nabla f) dm\nonumber\\
   \leq& B\sum_{i=1}^k(\Gamma_{k+1} - \Gamma_i)^2w_i +\frac{1}{B}\sum_{i=1}^k(\Gamma_{k+1} - \Gamma_i)\Big\|T(\nabla f, \nabla u_i ) +\frac{u_i}{2}\mathscr{L}f\Big\|_{L^2}^2 \nonumber\\
   &-B\sum_{i,j=1}^k(\Gamma_{k+1} - \Gamma_i)\Big((\Gamma_i-\Gamma_j)a_{ij} - \frac{1}{B}b_{ij}\Big)^2\nonumber\\
   \leq&B\sum_{i=1}^k(\Gamma_{k+1} - \Gamma_i)^2w_i +\frac{1}{B}\sum_{i=1}^k(\Gamma_{k+1} - \Gamma_i)\Big\|T(\nabla f, \nabla u_i ) +\frac{u_i}{2}\mathscr{L}f\Big\|_{L^2}^2.
\end{align}
Notice that,
\begin{align}\label{Equation-3.12}
    \int_\Omega fu_i T(\nabla f, \nabla \mathscr{L}u_i)dm + \int_\Omega fu_i \mathscr{L}T(\nabla f, \nabla u_i)dm = -\int_\Omega u_i \mathscr{L}u_i T(\nabla f, \nabla f)dm\nonumber\\
    -\int_\Omega fu_i\mathscr{L}u_i \mathscr{L}f dm + \int_\Omega u_i\mathscr{L}fT(\nabla f, \nabla u_i)dm + 2 \int_\Omega \Big(T(\nabla f, \nabla u_i)\Big)^2dm,
\end{align}
and
\begin{align}\label{Equation-3.13}
     -\int_\Omega hu_i\mathscr{L}u_i \mathscr{L}f dm + \int_\Omega fu_i \mathscr{L}(u_i\mathscr{L}f)dm = \int_\Omega u_i^2 (\mathscr{L}f)^2dm +\int_\Omega u_i \mathscr{L}fT(\nabla f, \nabla u_i)dm,
\end{align}
then, from \eqref{Equation-3.12} and \eqref{Equation-3.13} we have
\begin{align}\label{Equation-3.14}
   w_i =& \int_\Omega fu_ip_i dm\nonumber\\
   =&\int_\Omega \Big[u_i^2(\mathscr{L}f)^2 + 4\Big((T(\nabla f, \nabla u_i))^2 + u_i\mathscr{L}f T(\nabla f, \nabla u_i)\Big)-2u_i\mathscr{L}u_iT(\nabla f, \nabla f)\Big]dm.
\end{align}
Thus, substituting \eqref{Equation-3.14} into \eqref{Equation-3.11} we conclude the proof of Inequality~\eqref{lemma-i)-estimate}.

In order to proof Inequality~\eqref{lemma-ii)-estimate}, let us define 
\begin{align*}
    \psi_i = (h_i - a_i)u_1,
\end{align*}
where $a_i = \int_\Omega h_i u_1^2 dm$, hence $\int_\Omega \psi_i u_1 dm = 0$. By hypothesis $\int_\Omega h_i u_1 u_{j+1}dm = 0$, for $1 \leq j < i$, then we have
\begin{align*}
    \int_\Omega \psi_i u_{j+1} dm = 0, \quad \mbox{for} \quad 1 \leq j< i, \mbox{and} \quad \psi_i|_{\partial \Omega} = \frac{\partial \psi_i}{\partial \nu_T}|_{\partial \Omega}=0.
\end{align*}
Therefore, from Rayleigh-Ritz inequality, we have
\begin{align}\label{Rayleigh-Ritz}
    \Gamma_{i+1}\|\psi_i\|_{L^2}^2 \leq \int_\Omega \psi_i\mathscr{L}^2\psi_i dm .
\end{align}
By the definition of $\psi_i$, we get
\begin{align*}
    \mathscr{L}\psi_i = \mathscr{L}(h_i u_1) - a_i \mathscr{L}u_1= u_1 \mathscr{L} h_i + 2 T(\nabla h_i, \nabla u_1) + h_i \mathscr{L}u_1 - a_i\mathscr{L}u_1,
\end{align*}
hence
\begin{align*}
    \mathscr{L}^2 \psi_i =& \mathscr{L}\big( u_1 \mathscr{L} h_i + 2 T(\nabla h_i, \nabla u_1) + h_i \mathscr{L}u_1 - a_i\mathscr{L}u_1\big)\\
   =& u_1 \mathscr{L}^2h_i + 2 T(\nabla u_1, \nabla \mathscr{L} h_i) + 2\mathscr{L}u_1 \mathscr{L}h_i +2\mathscr{L}T(\nabla h_i, \nabla u_1) + 2T(\nabla h_i, \nabla \mathscr{L}u_1)\\ 
   &+ h_i\mathscr{L}^2u_1 - a_i\mathscr{L}^2u_1\\
   =& u_1 \mathscr{L}^2h_i + 2 T(\nabla u_1, \nabla \mathscr{L} h_i) + 2\mathscr{L}u_1 \mathscr{L}h_i +2\mathscr{L}T(\nabla h_i, \nabla u_1) + 2T(\nabla h_i, \nabla \mathscr{L}u_1) + \Gamma_1\psi_i.
\end{align*}
From \eqref{Rayleigh-Ritz} and the above equality, we get
\begin{align}\label{Equation-3.17}
    (\Gamma_{i+1} - \Gamma_1)\|\psi_i\|_{L^2}^2 \leq \int_\Omega \psi_i\wp_i dm = \int_\Omega \wp_i h_i u_1 dm - a_i\int_\Omega \wp_i u_1 dm,
\end{align}
where 
\begin{align*}
    \wp_i = u_1 \mathscr{L}^2h_i + 2 T(\nabla u_1, \nabla \mathscr{L} h_i) + 2\mathscr{L}u_1 \mathscr{L}h_i +2\mathscr{L}T(\nabla h_i, \nabla u_1) + 2T(\nabla h_i, \nabla \mathscr{L}u_1).
\end{align*}
Using divergence theorem, we have
\begin{align}\label{Equation-(3.18)}
    \int_\Omega \wp_i u_1 dm = 0,
\end{align}
and the following equalities
\begin{align}\label{Equation-3.18}
    2 \int_\Omega h_i u_1T(\nabla u_1, \nabla \mathscr{L} h_i) dm = \int_\Omega \Big(2u_1\mathscr{L}h_i T(\nabla u_1, \nabla h_i) + u_1^2(\mathscr{L}h_i)^2 - h_i u_1^2\mathscr{L}^2 h_i\Big)dm,
\end{align}
\begin{align}
    2 \int_\Omega h_i u_1& \mathscr{L}T(\nabla h_i, \nabla u_1) dm \nonumber\\
    =& \int_\Omega \Big(2u_1\mathscr{L}h_i T(\nabla u_1, \nabla h_i) + 2h_i \mathscr{L}u_1 T(\nabla h_i, \nabla u_1) + 4\big(T(\nabla h_i, \nabla u_1)\big)^2  \Big)dm,
\end{align}
\begin{align}\label{Equation-3.19}
    2 \int_\Omega  h_i u_1& T(\nabla h_i, \nabla \mathscr{L}u_1) dm\nonumber\\
    =& - 2\int_\Omega \Big(u_1\mathscr{L}u_1T(\nabla h_i, \nabla h_i) + h_i\mathscr{L}u_1T(\nabla u_1, \nabla h_i)+h_i u_1 \mathscr{L}h_i \mathscr{L}u_1\Big)dm.
\end{align}
Combining \eqref{Equation-3.18}-\eqref{Equation-3.19}, we obtain
\begin{align}\label{Equation-3.22}
    \int_\Omega \wp_i h_i u_1 dm =& \int_\Omega h_i u_1^2 \mathscr{L}^2h_i dm + 2 \int_\Omega h_i u_1 T(\nabla u_1, \nabla \mathscr{L}h_i)dm + 2\int_\Omega h_i u_1 \mathscr{L}u_1 \mathscr{L}h_i dm \nonumber \\
    &+2\int_\Omega h_i u_1 \mathscr{L}T(\nabla h_i, \nabla u_1)dm + 2\int_\Omega h_i u_1 T(\nabla h_i, \nabla u_1)dm\nonumber\\
    =& \int_\Omega \Big(u_1^2(\mathscr{L}h_i)^2 + 4\big(T(\nabla h_i, \nabla u_1)\big)^2 + 4u_1 \mathscr{L}h_iT(\nabla u_1, \nabla h_i) \Big)dm \nonumber\\
    &-2\int_\Omega u_1 \mathscr{L}u_1 T(\nabla h_i, \nabla h_i)dm\nonumber\\
    =&\int_\Omega \Big(u_1\mathscr{L}h_i + 2T(\nabla h_i, \nabla u_1)\Big)^2 dm -2\int_\Omega u_1 \mathscr{L}u_1 T(\nabla h_i, \nabla h_i)dm.
\end{align}
Substituting \eqref{Equation-(3.18)} and \eqref{Equation-3.22} into \eqref{Equation-3.17}, we have
\begin{equation}\label{Equation-3.23}
    (\Gamma_{i+1} - \Gamma_1)\|\psi_i\|_{L^2}^2 \leq \int_\Omega \Big(u_1\mathscr{L}h_i + 2T(\nabla h_i, \nabla u_1)\Big)^2 dm -2\int_\Omega u_1 \mathscr{L}u_1 T(\nabla h_i, \nabla h_i)dm.
\end{equation}
On the other hand, using integration by parts \eqref{parts}, we obtain
\begin{align*}
    \int_\Omega \psi_i\big(u_1&\mathscr{L}h_i+ 2T(\nabla h_i, \nabla u_1)\big) dm\\
    &=  \int_\Omega \psi_i\big(\mathscr{L}(u_1h_i)-h_i\mathscr{L}u_1\big)dm = \int_\Omega \psi_i \mathscr{L}(u_1h_i)dm - \int_\Omega \psi_i h_i\mathscr{L}u_1dm\\
    &= -\int_\Omega T\big(\nabla \psi_i, \nabla(u_1 h_i)\big)dm + \int_\Omega T\big(\nabla(\psi_i h_i), \nabla u_1\big)dm\\
     &= -\int_\Omega \psi_i T(\nabla u_1 , \nabla h_i)dm  -\int_\Omega u_1^2 T(\nabla h_i , \nabla h_i)dm + \int_\Omega \psi_i T(\nabla h_i, \nabla u_1)dm,
\end{align*}
that is,
\begin{align}\label{Equation-3.24}
\int_\Omega \psi_i\big(u_1\mathscr{L}h_i+ 2T(\nabla h_i, \nabla u_1)\big) dm = -\int_\Omega u_1^2 T(\nabla h_i , \nabla h_i)dm.
\end{align}
Thus, from \eqref{Equation-3.23} and \eqref{Equation-3.24}, for any constant $B>0$, we get
\begin{align*}
   (\Gamma_{i+1}& - \Gamma_1)^{\frac{1}{2}} \int_\Omega u_1^2 T(\nabla h_i , \nabla h_i)dm \\
  &= (\Gamma_{i+1} - \Gamma_1)^{\frac{1}{2}}\int_\Omega -\psi_i\big(u_1\mathscr{L}h_i+ 2T(\nabla h_i, \nabla u_1)\big) dm\\
   &\leq \frac{B}{2}(\Gamma_{i+1} - \Gamma_1)\|\psi_i\|_{L^2}^2 + \frac{1}{2B}\int_\Omega \big(u_1\mathscr{L}h_i+ 2T(\nabla h_i, \nabla u_1)\big)^2 dm\\
   &\leq \Big(\frac{B}{2}+\frac{1}{2B}\Big)\int_\Omega \big(u_1\mathscr{L}h_i+ 2T(\nabla h_i, \nabla u_1)\big)^2 dm -B\int_\Omega u_1 \mathscr{L}u_1 T(\nabla h_i, \nabla h_i) dm.
\end{align*}
Hence, we conclude the item ii). This complete the proof of Lemma~\eqref{lemma1}.
\end{proof}

From Lemma~\ref{lemma1} and follows some steps of the proof of Proposition~2 in Gomes and Miranda~\cite{GomesMiranda} we obtain our keystone technical lemma.
\begin{lem}\label{lemma2}
Let $\Omega$ be a bounded domain in an $n$-dimensional complete Riemannian manifold $M$ isometrically immersed in $\mathbb{R}^m$, $\Gamma_i$ be the $i$-th eigenvalue of Problem~\eqref{problem1} and $u_i$ be a normalized real-valued eigenfunction corresponding to $\Gamma_i$. Then we have
\begin{align}\label{Equation-3.3}
   \sum_{i=1}^k (\Gamma_{k+1}-\Gamma_i)^2 \int_\Omega u_i^2T(\nabla f, \nabla f)dm\leq& 4B\sum_{i=1}^k(\Gamma_{k+1}-\Gamma_i)^2\Big(q_i - \frac{1}{2}\int_\Omega u_i \mathscr{L}u_i \tr{(T)}dm\Big)\nonumber\\
   &+\frac{1}{B} \sum_{i=1}^k (\Gamma_{k+1}-\Gamma_i)q_i,
\end{align}
for any positive constant $B$, where
\begin{align}\label{Equation-qi}
  q_i=&\|T (\nabla u_i)\|^2_{L^2} +\frac{n^2}{4}\int_{\Omega}u_i^2|{\bf H}_T|^2dm+\int_\Omega u_i^2\Big(\frac{1}{4}|T(\nabla \eta)|^2+\frac{1}{2}\dv_\eta (T^2(\nabla \eta)) \Big)dm\nonumber\\
   & + \int_\Omega u_i T(\tr(\nabla T), \nabla u_i)dm +\frac{1}{4}\int_\Omega u_i^2\langle\tr(\nabla T), \tr(\nabla T) - 2T(\nabla \eta)\rangle dm,
\end{align}
and ${\bf H}_T$ is the generalized mean curvature vector of the immersion.
\end{lem}
\begin{proof} Let $x=(x_1, \ldots, x_m)$ be the position vector of the immersion of $M$ in $\mathbb{R}^m$, then taking $f=x_r$ in Lemma~\ref{lemma1} and summing over $\ell$ from 1 to $m$, we obtain
\begin{align}\label{eq4.17}
    \sum_{i=1}^k&(\Gamma_{k+1}-\Gamma_i)^2\int_{\Omega}\sum_{r=1}^mT(\nabla x_r, \nabla x_r)u_i^2dm \leq B\sum_{i=1}^k(\Gamma_{k+1}-\Gamma_i)^2 \Bigg\{\int_\Omega \sum_{r=1}^m u_i^2(\mathscr{L}x_r)^2 dm   \nonumber\\
    & + \int_\Omega \sum_{r=1}^m\Big[4\Big((T(\nabla x_r, \nabla u_i))^2 + u_i\mathscr{L}x_r T(\nabla x_r, \nabla u_i)\Big)-2u_i\mathscr{L}u_iT(\nabla x_r, \nabla x_r)\Big]dm\Bigg\}\nonumber\\
    &+\frac{1}{B}\sum_{i=1}^k(\Gamma_{k+1}-\Gamma_i)\int_{\Omega}\sum_{r=1}^m\Bigg[\frac{u_i^2}{4}(\mathscr{L} x_r)^2 + u_i \mathscr{L} x_r T(\nabla x_r, \nabla u_i) + |T(\nabla x_r, \nabla u_i)|^2 \Bigg]dm.
\end{align}
Let $\{e_1, \ldots, e_m\}$ be a local orthonormal geodesic frame in $p\in M$ adapted to $M$. Thus, by straightforward computation, similarly to the calculations in~\cite[Eq.~(3.17)-(3.24)]{GomesMiranda}, we obtain
\begin{align}\label{equation3.2}
      \quad \quad \sum_{r=1}^mT(\nabla x_r, \nabla x_r)=\sum_{r=1}^n\langle T(e_r), e_r \rangle = \tr(T),
\end{align}
\begin{align}\label{equation3.3}
      \sum_{r=1}^m|T(\nabla x_r, \nabla u_i)|^2= \sum_{r=1}^m | T(e_r, \nabla u_i)|^2 =|T(\nabla u_i)|^2,
\end{align} 
\begin{align}
    \sum_{r=1}^m (\mathscr{L} x_r)^2&= n^2|{\bf H}_T|^2+|\tr(\nabla T) - T(\nabla \eta)|^2\nonumber\\
    &=n^2|{\bf H}_T|^2+\langle\tr(\nabla T), \tr(\nabla T) - 2T(\nabla \eta)\rangle+|T(\nabla \eta)|^2,
\end{align}
and
\begin{align}\label{equation3.5}
    &\sum_{r=1}^m\mathscr{L} x_r T(\nabla x_\ell, \nabla u_i)=T(\tr(\nabla T), \nabla u_i) - T(T(\nabla \eta), \nabla u_i).
\end{align}
where ${\bf H}_T=\frac{1}{n}\sum_{i,j=1}^nT(e_i, e_j)\alpha(e_i, e_j)$ and $\tr(\nabla T):=\sum_{i=1}^n(\nabla T)(e_i, e_i)$.
Substituting \eqref{equation3.2}-\eqref{equation3.5} into \eqref{eq4.17} we get
\begin{align}\label{eq4.18}
    \sum_{i=1}^k&(\Gamma_{k+1}-\Gamma_i)^2 \int_\Omega \tr(T)u_i^2dm \leq B\sum_{i=1}^k(\Gamma_{k+1}-\Gamma_i)^2 \Bigg\{4\|T (\nabla u_i)\|_{L^2} +n^2\int_{\Omega}u_i^2|{\bf H}_T|^2dm \nonumber\\
     &+ 4\int_\Omega u_i T(\tr(\nabla T), \nabla u_i)dm + \int_\Omega u_i^2|T(\nabla \eta)|^2dm- 4\int_\Omega u_i  T(T(\nabla \eta), \nabla u_i)dm \nonumber\\
    &  +\int_\Omega u_i^2\langle\tr(\nabla T), \tr(\nabla T) - 2T(\nabla \eta)\rangle dm -2\int_\Omega u_i\mathscr{L}u_i\tr{(T)}dm \Bigg\}\nonumber\\
    &+\frac{1}{B}\sum_{i=1}^k(\Gamma_{k+1}-\Gamma_i)\Bigg\{\|T (\nabla u_i)\|_{L^2} +\frac{n^2}{4}\int_{\Omega}u_i^2|{\bf H}_T|^2dm+ \int_\Omega u_i T(\tr(\nabla T), \nabla u_i)dm \nonumber\\
   &+\frac{1}{4}\int_\Omega u_i^2|T(\nabla \eta)|^2dm- \int_\Omega u_i  T(T(\nabla \eta), \nabla u_i)dm +\frac{1}{4}\int_\Omega u_i^2\langle\tr(\nabla T), \tr(\nabla T) - 2T(\nabla \eta)\rangle dm \Bigg\}.
\end{align}
Since $u_i|_{\partial \Omega}=\frac{\partial u_i}{\partial \nu}_{\partial \Omega}=0$ from divergence theorem, we have
\begin{align*}
    -&\int_\Omega u_i T(T(\nabla \eta),\nabla u_i) dm = -\int_\Omega u_i \langle T^2(\nabla \eta), \nabla u_i\rangle dm \\
   &= -\frac{1}{2}\int_\Omega \langle T^2(\nabla \eta), \nabla u_i^2\rangle dm =\frac{1}{2}\int_\Omega u_i^2 \dv_\eta (T^2(\nabla \eta))dm.
\end{align*}
Substituting the previous equality into Inequality~\eqref{eq4.18} we complete the proof of Lemma~\ref{lemma2}. 
\end{proof}

\section{Proof of the main results}
Now, we are in a position to give the proof of theorems of this paper. 
\subsection{Proof of Theorem~\ref{theorem1.1}}
\begin{proof}
The proof of the first inequality is a consequence of Lemma~\ref{lemma2}. Since, from \eqref{Equation-2.1}, $|T(\nabla \eta)|\leq \delta|\nabla \eta|$, let us denote $T_0=\sup_{\Omega}|\tr(\nabla T)|$ and $\eta_0=\sup_{\Omega}|\nabla \eta|$ so that 
\begin{align}\label{Equation-4.1}
    \frac{1}{4}&\int_{\Omega}u_i^2\langle \tr(\nabla T), \tr(\nabla T) -2 T(\nabla \eta) \rangle dm\nonumber\\
    =& \frac{1}{4} \int_{\Omega} u_i^2 | \tr(\nabla T)|^2dm - \frac{1}{2}\int_{\Omega}u_i^2\langle \tr(\nabla T), T(\nabla \eta) \rangle dm\nonumber\\
    \leq& \frac{1}{4}T_0^2\int_\Omega u_i^2 dm +\frac{1}{2}\int_{\Omega}u_i^2|\tr(\nabla T)| |T(\nabla \eta)| dm \leq  \frac{T_0^2}{4}+\frac{\delta T_0\eta_0}{2}.
\end{align}
Furthermore,
\begin{align}\label{eqq31}
    \int_\Omega u_i T(\tr(\nabla T), \nabla u_i) dm &\leq \Big(\int_\Omega u_i^2dm \Big)^{\frac{1}{2}} \Big(\int_\Omega |T(\tr(\nabla T), \nabla u_i)|^2dm \Big)^{\frac{1}{2}} \nonumber\\
    & \leq T_0 \Big(\int_\Omega |T(\nabla u_i)|^2 dm\Big)^\frac{1}{2} = T_0\|T(\nabla u_i)\|_{L^2},
\end{align}
and 
\begin{equation}\label{inequation-HT}
    \frac{n^2}{4}\int_\Omega u_i^2|{\bf H}_T|^2dm \leq \frac{n^2H_0^2}{4}\int_\Omega u_i^2dm = \frac{n^2H_0^2}{4},
\end{equation}
where $H_0= \sup_{\Omega}|{\bf H}_T|$. Since $T$ is a symmetric positive $(1,1)$-tensor in the bounded domain $\Omega$, then there exist positive real numbers $\varepsilon$ and $\delta$ such that $\varepsilon I \leq T \leq \delta I$, consequently $n\varepsilon \leq \tr{(T)} \leq n\delta$, hence
\begin{align}
    n\varepsilon = n\varepsilon \int_\Omega u_i^2dm \leq \int_\Omega \tr{(T)}u_i^2 dm,
\end{align}
\begin{align}\label{Equation-4.5}
    \int_\Omega u_i^2\Big(\frac{1}{4}|T(\nabla \eta)|^2+\frac{1}{2}\dv_\eta (T^2(\nabla \eta)) \Big)dm &\leq C_0\int_\Omega u_i^2dm = C_0,
\end{align}
where $C_0=\sup_\Omega \Big\{\frac{1}{4}|T(\nabla \eta)|^2+\frac{1}{2}\dv_\eta (T^2(\nabla \eta)) \Big\} = \sup_\Omega \Big\{\frac{1}{2}\dv (T^2(\nabla \eta)) - \frac{1}{4}|T(\nabla \eta)|^2\Big\}$, and
\begin{align}\label{Equation-4.6}
    -\frac{1}{2}\int_\Omega u_i \mathscr{L}u_i \tr{(T)} dm &\leq \frac{1}{2} \Bigg(\int_\Omega u_i^2 dm \Bigg)^{\frac{1}{2}}\Bigg(\int_\Omega (\tr{(T))^2(\mathscr{L}u_i)^2}dm\Bigg)^{\frac{1}{2}}\nonumber\\
    &\leq \frac{n\delta}{2}\Bigg(\int_\Omega (\mathscr{L}u_i)^2dm\Bigg)^{\frac{1}{2}} = \frac{n\delta}{2}\Gamma_i^{\frac{1}{2}}.
\end{align}
From \eqref{Equation-qi}, \eqref{Equation-4.1}-\eqref{inequation-HT} and \eqref{Equation-4.5}, we get
\begin{align}\label{Equation-4.7}
    q_i &\leq \|T(\nabla u_i)\|_{L^2}^2   + T_0\|T(\nabla u_i)\|_{L^2} + \frac{T_0^2}{4} + \frac{\delta T_0\eta_0}{2}+C_0 +\frac{n^2H_0^2}{4} \nonumber\\
    &= \Big(\|T(\nabla u_i)\|_{L^2}  + \frac{1}{2}T_0\Big)^2 + \frac{n^2H_0^2+4C_0+2\delta T_0\eta_0}{4}.
\end{align}
Moreover, notice that
\begin{align*}
    \|T(\nabla u_i)\|_{L^2}^2 = \int_\Omega \langle T(\nabla u_i), T \nabla u_i \rangle dm \leq \delta \int_\Omega T(\nabla u_i, \nabla u_i) = - \delta \int_\Omega u_i \mathscr{L}u_i dm,
\end{align*}
therefore, from \eqref{lambda_i} we have
\begin{align}\label{Equation(4.8)}
    \|T(\nabla u_i)\|_{L^2}^2\leq - \delta \int_\Omega u_i \mathscr{L}u_i dm \leq \delta \Bigg( \int_\Omega u_i^2 dm\Bigg)^{\frac{1}{2}}\Bigg( \int_\Omega (\mathscr{L}u_i)^2 dm\Bigg)^{\frac{1}{2}} = \delta \Gamma_i^{\frac{1}{2}}.
\end{align}
Thus, from \eqref{Equation-4.7} and \eqref{Equation(4.8)} we obtain
\begin{align}\label{Equation-4.8}
    q_i &\leq \Big(\sqrt{\delta \Gamma^{\frac{1}{2}}}  + \frac{1}{2}T_0\Big)^2 + \frac{n^2H_0^2+4C_0+2\delta T_0\eta_0}{4}.
\end{align}
Substituting \eqref{Equation-4.6} and \eqref{Equation-4.8} into \eqref{Equation-3.3}, we have
\begin{align}\label{equation5-3}
     n\varepsilon&\sum_{i=1}^k(\Gamma_{k+1}-\Gamma_i)^2 \nonumber\\
     \leq& 4B\sum_{i=1}^k(\Gamma_{k+1}-\Gamma_i)^2\Bigg[\Big(\sqrt{\delta \Gamma^{\frac{1}{2}}}  + \frac{1}{2}T_0\Big)^2 + \frac{n\delta \Gamma^{\frac{1}{2}}}{2} + \frac{n^2H_0^2+4C_0+2\delta T_0\eta_0}{4}\Bigg]\nonumber\\
     & + \frac{1}{B}\sum_{i=1}^k(\Gamma_{k+1}-\Gamma_i)\Bigg[\Big(\sqrt{\delta \Gamma^{\frac{1}{2}}}  + \frac{1}{2}T_0\Big)^2 + \frac{n^2H_0^2+4C_0+2\delta T_0\eta_0}{4}\Bigg].
\end{align}
Finally, taking
\begin{align*}
    B=\frac{\Big\{\sum_{i=1}^k(\Gamma_{k+1}-\Gamma_i)\Big[\Big(\sqrt{\delta \Gamma^{\frac{1}{2}}}  + \frac{1}{2}T_0\Big)^2 + \frac{n^2H_0^2+4C_0+2\delta T_0\eta_0}{4}\Big]\Big\}^{\frac{1}{2}}}{\Big\{\sum_{i=1}^k(\Gamma_{k+1}-\Gamma_i)^2\Big[\Big(\sqrt{\delta \Gamma^{\frac{1}{2}}}  + \frac{1}{2}T_0\Big)^2 + \frac{n\delta \Gamma^{\frac{1}{2}}}{2} + \frac{n^2H_0^2+4C_0+2\delta T_0\eta_0}{4}\Big]\Big\}^{\frac{1}{2}}},
\end{align*}
into previous inequality, we complete the proof of Inequality~\eqref{theorem1-estimate1}.

Now, we give the proof of Inequality~\eqref{theorem1-estimate2} as a consequence of Lemma~\ref{lemma1}, item ii). Let $x=(x_1, \ldots, x_m)$ be the position vector of the immersion of $M$ in $\mathbb{R}^m$. Let us define an $m\times m$-matrix $B:=(b_{ij})_{m\times m}$, where 
\begin{align*}
 b_{ij}=\int_\Omega x_iu_1u_{j+1}dm.   
\end{align*}
 From the orthogonalization of Gram Schmidt, there exists an upper triangle matrix $R=(r_{ij})_{m \times m}$ and an orthogonal matrix $Q=(q_{ij})_{m \times m}$ such that $R=QB$, namely 
\begin{equation*}
    r_{ij}=\sum_{k=1}^m q_{ik}b_{kj}= \sum_{k=1}^m q_{ik} \int_\Omega x_k u_1  u_{j+1} dm = \int_\Omega \Big( \sum_{k=1}^m q_{ik}x_k\Big) u_1 u_{j+1} dm = 0,
\end{equation*}
for $1 \leq j < i \leq n$. Putting $h_i=\sum_{k=1}^m q_{ik}x_k$, we have 
\begin{equation*}
    \int_\Omega h_i u_1 u_{j+1} dm = 0   \quad \mbox{for} \quad 1 \leq j < i \leq m.
\end{equation*}
Therefore, we can apply the item~ii) of Lemma~\ref{lemma1} to obtain
\begin{align*}
       (\Gamma_{i+1}-\Gamma_1)^{\frac{1}{2}}\int_\Omega u_1^2 T(\nabla h_i, \nabla h_i) dm \leq& \Big(\frac{B}{2} + \frac{1}{2B}\Big)\int_\Omega(u_1 \mathscr{L}h_i + 2T(\nabla h_i, \nabla u_1))^2dm \nonumber\\
       &- B \int_\Omega u_1 \mathscr{L}u_1 T(\nabla h_i, \nabla h_i)dm.
\end{align*}
Summing over $i$ from $1$ to $m$ in the above inequality
\begin{align}\label{Equation-4.10}
        \sum_{i=1}^m&(\Gamma_{i+1}-\Gamma_1)^{\frac{1}{2}}\int_\Omega u_1^2 T(\nabla h_i, \nabla h_i) dm \nonumber\\
       &\leq \Big(\frac{B}{2} + \frac{1}{2B}\Big) \sum_{i=1}^m\int_\Omega(u_1 \mathscr{L}h_i + 2T(\nabla h_i, \nabla u_1))^2dm - B  \sum_{i=1}^m\int_\Omega u_1 \mathscr{L}u_1 T(\nabla h_i, \nabla h_i)dm.
\end{align}
Since $h_i=\sum_{k=1}^m q_{ik}x_k$ and $Q$ is an orthogonal matrix, from \eqref{equation3.2}-\eqref{equation3.5}, we have
\begin{align*}
    \sum_{i=1}^m |\nabla h_i|^2=n, \quad \sum_{i=1}^m |\nabla h_iu_1|^2 = u_1^2, \quad \sum_{i=1}^m |T(\nabla h_i, \nabla u_1)|^2=|T(\nabla u_1)|^2,
\end{align*}
\begin{align*}
    \sum_{i=1}^m \mathscr{L}h_i T(\nabla h_i, \nabla u_1) = T(\tr{(\nabla T)}, \nabla u_1) - T(T(\nabla \eta), \nabla u_1),
\end{align*}
and
\begin{align*}
    \sum_{i=1}^m (\mathscr{L}h_i)^2 = n^2 |{\bf H}_T|^2 + |\tr{(\nabla T)}|^2 - 2 \langle \tr{(\nabla T)}, T(\nabla \eta)\rangle + |T(\nabla \eta)|^2.
\end{align*}
Since  $T_0=\sup_{\Omega}|\tr(\nabla T)|$ and $\eta_0=\sup_{\Omega}|\nabla \eta|$, from \eqref{Equation-2.1}, \eqref{Equation-4.10} and the previous equalities, we obtain
\begin{align}\label{Equation-4.12}
\sum_{i=1}^m&\int_\Omega(u_1 \mathscr{L}h_i + 2T(\nabla h_i, \nabla u_1))^2dm \nonumber\\
=& \sum_{i=1}^m\int_\Omega\Big(u_1^2 \big(\mathscr{L}h_i\big)^2 + 4u_1\mathscr{L}h_iT(\nabla h_i, \nabla u_1) + 4\big(T(\nabla h_i, \nabla u_1)\big)^2\Big)dm\nonumber\\
=&\int_\Omega\Big(u_1^2 \big(n^2 |{\bf H}_T|^2 + |\tr{(\nabla T)}|^2 - 2 \langle \tr{(\nabla T)}, T(\nabla \eta)\rangle + |T(\nabla \eta)|^2) \nonumber\\
&+ 4u_1\big(T(\tr{(\nabla T)}, \nabla u_1) - T(T(\nabla \eta), \nabla u_1)\big) + 4|T(\nabla u_1)|^2\Big)dm\nonumber\\
\leq& 4\|T(\nabla u_1)\|_{L^2}^2 + 4T_0\|T(\nabla u_1)\|_{L^2} + T_0^2 + n^2H_0^2 + 2\delta T_0\eta_0 + 4C_0\nonumber\\
=& \big(2\|T(\nabla u_1)\|_{L^2}+T_0\big)^2 + n^2H_0^2 + 2\delta T_0\eta_0 + 4C_0,
\end{align}
where $H_0= \sup_{\Omega}|{\bf H}_T|$ and $C_0=\sup_\Omega \Big\{\frac{1}{2}\dv (T^2(\nabla \eta)) - \frac{1}{4}|T(\nabla \eta)|^2\Big\}$. Similar to \eqref{Equation(4.8)} we have
\begin{align}\label{Equation-4.13}
    \|T(\nabla u_1)\|_{L^2}^2 \leq \delta \Gamma_1^{\frac{1}{2}},
\end{align}
Moreover, by \eqref{lambda_i} and since $\tr{(T)} \leq n\delta$, we have
\begin{align}\label{Equation-4.14}
    - \sum_{i=1}^m\int_\Omega u_1 \mathscr{L}u_1T(\nabla h_i, \nabla h_i)dm &= - \int_\Omega u_1 \mathscr{L}u_1 \tr{(T)} dm \nonumber\\
    &\leq n\delta \Big(\int_\Omega (\mathscr{L}u_1)^2 dm\Big)^{\frac{1}{2}} = n \delta \Gamma_1^{\frac{1}{2}},
\end{align}
and
\begin{align}\label{Equation-4.15}
    \sum_{i=1}^m(\Gamma_{i+1}-\Gamma_1)^{\frac{1}{2}}\int_\Omega u_1^2 T(\nabla h_i, \nabla h_i) dm \geq  \varepsilon \sum_{i=1}^m(\Gamma_{i+1}-\Gamma_1)^{\frac{1}{2}}|\nabla h_i|^2.
\end{align}
Substituting \eqref{Equation-4.12}, \eqref{Equation-4.13}, \eqref{Equation-4.14} and \eqref{Equation-4.15} into \eqref{Equation-4.10}, we have
\begin{align}\label{Equation-4.16}
         \varepsilon&\sum_{i=1}^m(\Gamma_{i+1}-\Gamma_1)^{\frac{1}{2}}|\nabla h_i|^2 \nonumber\\
       &\leq \Big(\frac{B}{2} + \frac{1}{2B}\Big) \Big(\big(2\sqrt{\delta \Gamma_1^{\frac{1}{2}}}+T_0\big)^2 + n^2H_0^2 + 2\delta T_0\eta_0 + 4C_0\Big) + B n \delta \Gamma_1^{\frac{1}{2}}.
\end{align}
Notice that,
\begin{align*}
    \sum_{i=1}^m(\Gamma_{i+1} - \Gamma_1)^{\frac{1}{2}}|\nabla h_i|^2&\geq \sum_{i=1}^n(\Gamma_{i+1} - \Gamma_1)^{\frac{1}{2}}|\nabla h_i|^2+(\Gamma_{n+1} - \Gamma_1)^{\frac{1}{2}}\sum_{\gamma=n+1}^m|\nabla h_\gamma|^2\\ 
    &= \sum_{i=1}^n(\Gamma_{i+1} - \Gamma_1)^{\frac{1}{2}}|\nabla h_i|^2+(\Gamma_{n+1}-\Gamma_1)^{\frac{1}{2}}\Big(n-\sum_{i=1}^n|\nabla h_i|^2\Big)\\
    &=\sum_{i=1}^n(\Gamma_{i+1} - \Gamma_1)^{\frac{1}{2}}|\nabla h_i|^2+(\Gamma_{n+1}-\Gamma_1)^{\frac{1}{2}}\sum_{i=1}^n(1-|\nabla h_i|^2)\\
    &\geq \sum_{i=1}^n(\Gamma_{i+1} - \Gamma_1)^{\frac{1}{2}}|\nabla h_i|^2+\sum_{i=1}^n(\Gamma_{i+1}-\Gamma_1)^{\frac{1}{2}}(1-|\nabla h_i|^2)\\
    &=\sum_{i=1}^n(\Gamma_{i+1}-\Gamma_1)^{\frac{1}{2}}.
\end{align*}
From the previous inequality and \eqref{Equation-4.16} we get
\begin{align*}
         \varepsilon&\sum_{i=1}^n(\Gamma_{i+1}-\Gamma_1)^{\frac{1}{2}} \nonumber\\
       &\leq \Big(\frac{B}{2} + \frac{1}{2B}\Big) \Bigg(\Big(2\sqrt{\delta \Gamma_1^{\frac{1}{2}}}+T_0\Big)^2 + n^2H_0^2 + 2\delta T_0\eta_0 + 4\delta^2C_0\Bigg) + Bn \delta \Gamma_1^{\frac{1}{2}},
\end{align*}
and taking
\begin{align*}
    B = \frac{\Bigg(\Big(2\sqrt{\delta \Gamma_1^{\frac{1}{2}}}+T_0\Big)^2 + n^2H_0^2 + 2\delta T_0\eta_0 + 4C_0\Bigg)^{\frac{1}{2}}}{\Bigg(\Big(2\sqrt{\delta \Gamma_1^{\frac{1}{2}}}+T_0\Big)^2 + n^2H_0^2 + 2\delta T_0\eta_0 + 4C_0+ 2n \delta \Gamma_1^{\frac{1}{2}}\Bigg)^{\frac{1}{2}}},
\end{align*}
into above inequality, we obtain \eqref{theorem1-estimate2} and complete the proof of Theorem~\ref{theorem1.1}.
\end{proof}

\section{Proof of Corollaries~\ref{cor-T-divergence-free2} and \ref{cor-T-divergence-free3}}
To prove Corollary~\ref{cor-T-divergence-free2}, we needed an algebraic lemma obtained by Jost et al.~\cite[Lemma~2.3]{Jost}.
\begin{lem}
Let $\{a_i\}_{i=1}^m$, $\{b_i\}_{i=1}^m$ and $\{c_i\}_{i=1}^m$ be three sequences of non-negative real with $\{a_i\}_{i=1}^m$ decreasing and $\{b_i\}_{i=1}^m$ and $\{c_i\}_{i=1}^m$ increasing. Then we have
\begin{equation*}
    \Big(\sum_{i=1}^m a_i^2b_i \Big)\Big(\sum_{i=1}^m a_ic_i \Big)\leq \Big(\sum_{i=1}^m a_i^2\Big)\Big(\sum_{i=1}^m a_ib_ic_i\Big).
\end{equation*}
\end{lem}
\subsection{Proof of Corollary~\ref{cor-T-divergence-free2}}
Since $\{\Lambda_i=(\Gamma_{k+1} - \Gamma_i)\}_{i=1}^k$ is decreasing and $\{(4+2n)\delta\Gamma_i^{\frac{1}{2}} + n^2H_0^2+4C_0\}_{i=1}^k$ and $\{4\delta \Gamma_i^{\frac{1}{2}}+ n^2H_0^2+4C_0\}$ are increasing, we obtain from the previous lemma that
\begin{align}\label{Equation-1.5}
    &\Bigg\{\sum_{i=1}^k\Lambda_i^2\Big[(4+2n)\delta\Gamma_i^{\frac{1}{2}} + n^2H_0^2+4C_0 \Big]\Bigg\}\Bigg\{\sum_{i=1}^k\Lambda_i\Big(4\delta \Gamma_i^{\frac{1}{2}}+ n^2H_0^2+4C_0 \Big)\Bigg\}\nonumber\\
    &\leq \Bigg\{\sum_{i=1}^k\Lambda_i^2\Bigg\}\Bigg\{\sum_{i=1}^k\Lambda_i\Big[(4+2n)\delta\Gamma_i^{\frac{1}{2}} + n^2H_0^2+4C_0 \Big]\Big(4\delta \Gamma_i^{\frac{1}{2}}+ n^2H_0^2+4C_0 \Big)\Bigg\}
\end{align}
Substituting \eqref{Equation-1.5} into \eqref{Equation-1.4}, we complete the proof of the corollary. 
\subsection{Proof of Corollary~\ref{cor-T-divergence-free3}}
\begin{proof}
Since our Inequality~\eqref{Equation-1.6} is a quadratic inequality of $\Gamma_{k+1}$, it is not difficult to obtain \eqref{Equation-1.7}, that is,
\begin{equation}\label{Equation-1.7-2}
    \Gamma_{k+1} \geq A_k + \sqrt{A_k^2-B_k}.
\end{equation}
Now, notice that Inequality~\eqref{Equation-1.6} also holds if we replace the integer $k$ with $k-1$, that is, we have
\begin{align*}
    \sum_{i=1}^{k-1}&(\Gamma_{k}-\Gamma_i)^2\nonumber \\
    &\leq\frac{1}{n^2\varepsilon^2}\sum_{i=1}^{k-1}(\Gamma_{k}-\Gamma_i)\Big[(4+2n)\delta\Gamma_i^{\frac{1}{2}} + n^2H_0^2+4C_0 \Big]\Big(4\delta \Gamma_i^{\frac{1}{2}}+ n^2H_0^2+4C_0 \Big),
\end{align*}
hence, we infer
\begin{align*}
    \sum_{i=1}^{k}&(\Gamma_{k}-\Gamma_i)^2\nonumber \\
    &\leq\frac{1}{n^2\varepsilon^2}\sum_{i=1}^{k}(\Gamma_{k}-\Gamma_i)\Big[(4+2n)\delta\Gamma_i^{\frac{1}{2}} + n^2H_0^2+4C_0 \Big]\Big(4\delta \Gamma_i^{\frac{1}{2}}+ n^2H_0^2+4C_0 \Big).
\end{align*}
Therefore, $\Gamma_k$ also satisfies the same quadratic inequality and we obtain
\begin{equation*}
    \Gamma_k \geq A_k - \sqrt{A_k^2-B_k}.
\end{equation*}
Thus, from the previous inequality and \eqref{Equation-1.7-2} we get \eqref{Equation-1.8} and complete the proof of Corollary~\ref{cor-T-divergence-free3}.
\end{proof}

\subsection{Proof of Theorem~\ref{theorem3}}
\begin{proof} i) We can see from \eqref{specialfunction-1} that
\begin{equation}\label{Equation-5.2}
    |\mathscr{L}\varphi| \leq \delta |\Delta_\eta \varphi|=\delta|\Delta \varphi - \langle \nabla \eta, \nabla \varphi \rangle| \leq \delta (\Delta \varphi +|\nabla \eta||\nabla \varphi|) \leq \delta (A_0 + \eta_0).
\end{equation}
Then, taking $f=\varphi$ into Lemma~\eqref{lemma1}, we have
\begin{align}\label{Equation-5.3}
    \varepsilon \sum_{i=1}^k&(\Gamma_{k+1}-\Gamma_i)^2 \leq \sum_{i=1}^k(\Gamma_{k+1}-\Gamma_i)^2 \int_\Omega T(\nabla \varphi, \nabla \varphi) u_i^2 dm \nonumber \\
    \leq& B\sum_{i=1}^k(\Gamma_{k+1}-\Gamma_i)^2 \nonumber\\
    &\times \int_\Omega \Big[u_i^2(\mathscr{L}\varphi)^2 + 4\Big((T(\nabla \varphi, \nabla u_i))^2 + u_i\mathscr{L}\varphi T(\nabla \varphi, \nabla u_i)\Big)-2u_i\mathscr{L}u_i T(\nabla \varphi, \nabla \varphi)\Big]dm\nonumber\\
    &+\frac{1}{B}\sum_{i=1}^k(\Gamma_{k+1}-\Gamma_i)\int_\Omega \Big(T(\nabla \varphi, \nabla u_i) + \frac{u_i}{2}\mathscr{L}\varphi \Big)^2dm.
\end{align}
From \eqref{Equation-2.1}, \eqref{lambda_i} and \eqref{Equation-5.2}, since $T \leq \delta I$ and using the Schwarz inequality, we have
\begin{align}\label{Equation-5.4}
    \int_\Omega T(\nabla \varphi, \nabla u_i)^2 dm =& \int_\Omega \langle \nabla \varphi, T(\nabla u_i) \rangle^2dm \leq \int_\Omega |\nabla \varphi|^2|T(\nabla u_i)|^2dm \nonumber\\
    =&  \int_\Omega |T(\nabla u_i)|^2dm \leq \delta \Gamma_i^\frac{1}{2},
\end{align}
\begin{align}
    \int_\Omega u_i \mathscr{L}\varphi T(\nabla \varphi, \nabla u_i) dm \leq& \int_\Omega |u_i||\mathscr{L}\varphi||\nabla \varphi||T(\nabla u_i)|dm \nonumber\\
    \leq& \delta(A_0 + \eta_0)\Bigg(\int_\Omega u_i^2 dm \Bigg)^{\frac{1}{2}}\Bigg(\int_\Omega |T(\nabla u_i)|^2 dm \Bigg)^{\frac{1}{2}}\nonumber\\
    \leq& \delta^{\frac{3}{2}}(A_0 + \eta_0)\Gamma_i^{\frac{1}{4}},
\end{align}
\begin{align}
    -2\int_\Omega u_i\mathscr{L}u_i T(\nabla \varphi, \nabla \varphi) dm &\leq 2 \delta \int_\Omega |u_i||\mathscr{L}u_i||\nabla \varphi|^2dm \nonumber\\
    &\leq 2 \delta\Bigg(\int_\Omega u_i^2 dm\Bigg)^\frac{1}{2} \Bigg(\int_\Omega (\mathscr{L}u_i)^2dm\Bigg)^{\frac{1}{2}}= 2\delta\Gamma_i^\frac{1}{2},
\end{align}
and
\begin{align}\label{Equation-5.7}
  \int_\Omega \Big(T(\nabla \varphi, \nabla u_i) + \frac{u_i}{2}\mathscr{L}\varphi \Big)^2dm =& \int_\Omega \big(\langle \nabla \varphi, T(\nabla u_i)\rangle^2 + \langle \nabla \varphi, T(\nabla u_i)\rangle u_i\mathscr{L}\varphi + \frac{u_i^2}{4}(\mathscr{L}\varphi)^2 \big) dm\nonumber\\
  \leq & \delta \Gamma_i^{\frac{1}{2}} + \delta^\frac{3}{2}(A_0 + \eta_0)\Gamma_i^\frac{1}{4} + \frac{\delta^2}{4}(A_0 + \eta_0)^2.
\end{align}
Substituting \eqref{Equation-5.4}-\eqref{Equation-5.7} into \eqref{Equation-5.3} we obtain
\begin{align*}
    \varepsilon \sum_{i=1}^k(\Gamma_{k+1}-\Gamma_i)^2 \leq& B \sum_{i=1}^k(\Gamma_{k+1} - \Gamma_i)^2\Big(6\delta\Gamma_i^\frac{1}{2} + 4\delta^\frac{3}{2}(A_0 + \eta_0)\Gamma_i^\frac{1}{4} + \delta^2(A_0 + \eta_0)^2\Big) \nonumber\\
    &+\frac{1}{B}\sum_{i=1}^k(\Gamma_{k+1} - \Gamma_i)\Big(\delta\Gamma_i^\frac{1}{2} + \delta^\frac{3}{2}(A_0 + \eta_0)\Gamma_i^\frac{1}{4} + \frac{\delta^2}{4}(A_0 + \eta_0)^2\Big).
\end{align*}
Taking 
\begin{align*}
    B = \frac{\Big\{\sum_{i=1}^k(\Gamma_{k+1} - \Gamma_i)\Big(\delta\Gamma_i^\frac{1}{2} + \delta^\frac{3}{2}(A_0 + \eta_0)\Gamma_i^\frac{1}{4} + \frac{\delta^2}{4}(A_0 + \eta_0)^2\Big)\Big\}^\frac{1}{2}}{\Big\{\sum_{i=1}^k(\Gamma_{k+1} - \Gamma_i)^2\Big(6\delta\Gamma_i^\frac{1}{2} + 4\delta^\frac{3}{2}(A_0 + \eta_0)\Gamma_i^\frac{1}{4} + \delta^2(A_0 + \eta_0)^2\Big) \Big\}^\frac{1}{2}},
\end{align*}
we get \eqref{thm3-i} and complete the proof of item i).

ii) Taking $f=f_\alpha$  into Lemma~\eqref{lemma1} and summing over $\alpha$, we get
\begin{align}\label{Equation-4.21}
    \gamma \varepsilon& \sum_{i=1}^k(\lambda_{k+1}-\lambda_i)^2 = \varepsilon \sum_{\alpha=1}^{m+1}\sum_i^k(\Gamma_{k+1} - \Gamma_i)^2 \int_\Omega u_i^2 |\nabla f_\alpha|^2 dm \nonumber\\
    \leq& B\sum_{\alpha=1}^{m+1}\sum_{i=1}^k(\Gamma_{k+1}-\Gamma_i)^2\nonumber\\
    &\times\int_\Omega \Big[u_i^2(\mathscr{L}f_\alpha)^2 + 4\Big((T(\nabla f_\alpha, \nabla u_i))^2 + u_i\mathscr{L}f_\alpha T(\nabla f_\alpha, \nabla u_i)\Big)-2u_i\mathscr{L}u_iT(\nabla f_\alpha, \nabla f_\alpha)\Big]dm\nonumber\\
    &+ \frac{1}{B}\sum_{\alpha=1}^{m+1}\sum_{i=1}^k(\Gamma_{k+1}-\Gamma_i)\int_\Omega \Big(T(\nabla f_\alpha, \nabla u_i) + \frac{u_i\mathscr{L}f_\alpha}{2}\Big)^2dm.\nonumber\\
\end{align}
Using \eqref{specialfunction-2} we have
\begin{equation*}
    \sum_{\alpha=1}^{m+1}f_\alpha \nabla f_\alpha = 0, \quad \mbox{and} \quad \sum_{\alpha=1}^{m+1}|\nabla f_\alpha|^2 = \gamma.
\end{equation*}
Hence, from the previous equalities and the Schwarz inequality, we obtain
\begin{align}\label{Equation-4.22}
    \sum_{\alpha=1}^{m+1}\mathscr{L}f_\alpha T(\nabla f_\alpha, \nabla u_i) &\leq \delta \sum_{\alpha =1}^{m+1} \Delta_\eta f_\alpha \langle\nabla f_\alpha, T(\nabla u_i)\rangle \nonumber\\
    &= \delta \sum_{\alpha =1}^{m+1} (\Delta f_\alpha - \langle \nabla \eta, \nabla f_\alpha \rangle)\langle \nabla  f_\alpha, T(\nabla u_i)\rangle\nonumber\\
    & \leq \delta \sum_{\alpha =1}^{m+1}(-\gamma f_\alpha\langle \nabla  f_\alpha, T(\nabla u_i)\rangle + |\nabla f_\alpha|^2|\nabla \eta||T(\nabla u_i)|)\nonumber\\
    &\leq \delta \gamma \eta_0 |T(\nabla u_i)|,
\end{align}
\begin{align}\label{Equation-4.23}
    \sum_{\alpha=1}^{m+1} (\mathscr{L}f_\alpha)^2 &\leq \delta^2 \sum_{\alpha=1}^{m+1}(\Delta_\eta f_\alpha)^2 = \delta^2\sum_{\alpha=1}^{m+1}(\Delta f_\alpha - \langle \nabla \eta, \nabla f_\alpha \rangle)^2\nonumber\\
    &=\delta^2\sum_{\alpha=1}^{m+1}((\Delta f_\alpha)^2 - 2\Delta f_\alpha \langle \nabla \eta, \nabla f_\alpha \rangle + \langle \nabla \eta, \nabla f_\alpha \rangle^2)\nonumber\\
    &\leq \delta^2(\gamma^2 +\eta_0 \gamma) = \gamma \delta^2(\gamma + \eta_0^2),
\end{align}
\begin{align}
    \sum_{\alpha=1}^{m+1} T(\nabla f_\alpha, \nabla u_i)^2=\sum_{\alpha=1}^{m+1} \langle \nabla f_\alpha, T(\nabla u_i)\rangle^2 \leq \sum_{\alpha=1}^{m+1} |\nabla f_\alpha|^2 |T(\nabla u_i)|^2 = \gamma |T(\nabla u_i)|^2,
\end{align}
and
\begin{align}\label{Equation-4.24}
    \sum_{\alpha=1}^{m+1}\Big(\langle \nabla f_\alpha, &T(\nabla u_i)\rangle + \frac{u_i}{2}\mathscr{L}f_\alpha\Big)^2\nonumber\\
    =& \sum_{\alpha=1}^{m+1}\Big(\langle \nabla f_\alpha, T(\nabla u_i)\rangle^2 + \langle \nabla f_\alpha, T(\nabla u_i)\rangle u_i \mathscr{L}f_\alpha +\frac{u_i^2}{4}(\mathscr{L}f_\alpha)^2 \Big)\nonumber\\
    \leq& \gamma|T(\nabla u_i)|^2 + \gamma \eta_0|u_i||T(\nabla u_i)| + \frac{\delta^2}{4}(\gamma^2 + \gamma\eta_0^2)u_i^2.
\end{align}
Substituting \eqref{Equation-4.22}-\eqref{Equation-4.24} into \eqref{Equation-4.21}, we get
\begin{align*}
    \gamma \varepsilon \sum_{i=1}^k(\Gamma_{k+1}-\Gamma_i)^2 \leq& \sum_{i=1}^k(\Gamma_{k+1}-\Gamma_i)^2\Big(6\gamma \delta \Gamma_i^\frac{1}{2} + 4\gamma\delta^\frac{3}{2}\eta_0\Gamma_i^\frac{1}{4} + \delta^2 \gamma(\gamma + \eta_0^2) \Big)\\
    &+ \frac{1}{B} \sum_{i=1}^k(\Gamma_{k+1}-\Gamma_i)\Big(\gamma \delta \Gamma_i^\frac{1}{2} + \gamma\delta^\frac{3}{2}\eta_0\Gamma_i^\frac{1}{4} + \frac{\delta^2}{4}\gamma(\gamma + \eta_0^2) \Big),
\end{align*}
and taking
\begin{align*}
    B = \frac{\Big\{\sum_{i=1}^k(\Gamma_{k+1}-\Gamma_i)\Big(\gamma \delta \Gamma_i^\frac{1}{2} + \gamma\delta^\frac{3}{2}\eta_0\Gamma_i^\frac{1}{4} + \frac{\delta^2}{4}\gamma(\gamma + \eta_0^2) \Big)\Big\}^\frac{1}{2}}{\Big\{\sum_{i=1}^k(\Gamma_{k+1}-\Gamma_i)^2\Big(6\gamma \delta \Gamma_i^\frac{1}{2} + 4\gamma\delta^\frac{3}{2}\eta_0\Gamma_i^\frac{1}{4} + \delta^2 \gamma(\gamma + \eta_0^2) \Big)\Big\}^\frac{1}{2}},
\end{align*}
into the previous inequality we obtain \eqref{thm3-ii} and complete the proof of Theorem~\ref{theorem3}.
\end{proof}

\section{Concluding remarks}\label{concludingremarks}
In this section, we can verify that some results by Du et al.\cite{Du-etal}, for the drifted Laplacian operator,  can be obtained through a slight modification of our Theorem~\ref{theorem1.1}. In fact, in more general configurations we have the following result.
\begin{theorem}\label{Generalization-Duetal-teo} Let $\Omega$ be a bounded domain in an $n$-dimensional complete Riemannian manifold $M^n$ isometrically immersed in $\mathbb{R}^m$. Denote by $\Gamma_i$ the $i$-th eigenvalue  of Problem~\eqref{problem1}. Then, we get
\begin{align*}
    \sum_{i=1}^k&(\Gamma_{k+1}-\Gamma_i)^2\nonumber\\
    \leq& \frac{4}{n\varepsilon}\Bigg\{\sum_{i=1}^k(\Gamma_{k+1}-\Gamma_i)^2\Big[\Big(\sqrt{\delta \Gamma_i^{\frac{1}{2}}}  + \frac{T_0}{2}\Big)^2 +\frac{n\delta \Gamma_i^{\frac{1}{2}}}{2} + \delta^2 \eta_0\Gamma_i^{\frac{1}{4}} + \frac{n^2H_0^2+\delta^2\eta_0^2+2\delta T_0\eta_0}{4} \Big]\Bigg\}^{\frac{1}{2}}\nonumber\\
    &\times\Bigg\{\sum_{i=1}^k(\Gamma_{k+1}-\Gamma_i)\Big[\Big(\sqrt{\delta \Gamma_i^{\frac{1}{2}}}  + \frac{T_0}{2}\Big)^2 + \delta^2 \eta_0\Gamma_i^{\frac{1}{4}} + \frac{n^2H_0^2+\delta^2\eta_0^2+2\delta T_0\eta_0}{4} \Big]\Bigg\}^{\frac{1}{2}},
\end{align*}
and
\begin{align*}
    \sum_{i=1}^k(\Gamma_{k+1}-\Gamma_1)^\frac{1}{2}\leq& \frac{1}{\varepsilon}\Bigg[4\Big(\sqrt{\delta \Gamma_1^{\frac{1}{2}}}  + \frac{T_0}{2}\Big)^2 + 4\delta^2 \eta_0\Gamma_1^{\frac{1}{4}} + n^2H_0^2+\delta^2\eta_0^2+2\delta T_0\eta_0\Bigg]^{\frac{1}{2}}\nonumber\\
    &\times\Bigg[4\Big(\sqrt{\delta \Gamma_1^{\frac{1}{2}}}  + \frac{T_0}{2}\Big)^2 + 2n\delta \Gamma_1^{\frac{1}{2}} + 4\delta^2 \eta_0\Gamma_1^{\frac{1}{4}} + n^2H_0^2+\delta^2\eta_0^2+2\delta T_0\eta_0\Bigg]^{\frac{1}{2}},
\end{align*}
where $T_0=\sup_{\Omega}|\tr(\nabla T)|$, $\eta_0=\sup_{\Omega}|\nabla \eta|$, $H_0=\sup_\Omega |{\bf H}_T|$ and ${\bf H}_T$ is the generalized mean curvature vector of the immersion. In particular, to Dirichlet problem for the bi-drifted Cheng-Yau operator,
\begin{align}\label{Equation-6.1}
    \sum_{i=1}^k(\Gamma_{k+1}-\Gamma_i)^2 \leq& \frac{1}{n\varepsilon}\Bigg\{\sum_{i=1}^k(\Gamma_{k+1}-\Gamma_i)^2\Big[(4+2n)\delta\Gamma_i^{\frac{1}{2}} + 4\delta^{\frac{3}{2}} \eta_0\Gamma_i^{\frac{1}{4}} + n^2H_0^2+\delta^2\eta_0^2 \Big]\Bigg\}^{\frac{1}{2}}\nonumber\\
    &\times\Bigg\{\sum_{i=1}^k(\Gamma_{k+1}-\Gamma_i)\Big(4\delta \Gamma_i^{\frac{1}{2}} + 4\delta^{\frac{3}{2}} \eta_0\Gamma_i^{\frac{1}{4}} + n^2H_0^2+\delta^2\eta_0^2 \Big)\Bigg\}^{\frac{1}{2}},
\end{align}
and
\begin{align}\label{Equation-6.2}
    \sum_{i=1}^k&(\Gamma_{k+1}-\Gamma_1)^\frac{1}{2}\nonumber\\
    &\leq \frac{1}{\varepsilon}\Bigg[\Big(4\delta \Gamma_1^{\frac{1}{2}} + 4\delta^{\frac{3}{2}} \eta_0\Gamma_1^{\frac{1}{4}} + n^2H_0^2+\delta^2\eta_0^2\Big)\Big((4 + 2n)\delta \Gamma_i^{\frac{1}{2}} + 4\delta^{\frac{3}{2}} \eta_0\Gamma_1^{\frac{1}{4}} + n^2H_0^2+\delta^2\eta_0^2\Big)\Bigg]^{\frac{1}{2}}.
\end{align}
\end{theorem}
\begin{proof}
The proof is a slight modification from proof of Theorem~\ref{theorem1.1}. In fact, from \eqref{Equation-2.1} and Inequality~\eqref{Equation-4.5}, we can see that
\begin{align}\label{Equation-4.5(2)}
    \int_\Omega u_i^2\Big(\frac{1}{4}|T(\nabla \eta)|^2+\frac{1}{2}&\dv_\eta (T^2(\nabla \eta)) \Big)dm\nonumber\\
    &= \int_\Omega u_i^2\Big(\frac{1}{2}\dv_\eta (T^2(\nabla \eta)) +\frac{1}{4}|T(\nabla \eta)|^2 \Big)dm\nonumber\\
     &= - \int_\Omega u_i\langle T(\nabla \eta), T(\nabla u_i)\rangle dm + \frac{1}{4}\int_\Omega u_i^2|T(\nabla \eta)|^2dm \nonumber \\
     &\leq \Big(\int_\Omega |T(\nabla \eta)|^2|T(\nabla u_i)|^2dm\Big)^{\frac{1}{2}} + \frac{\delta^2\eta_0^2}{4}\nonumber\\
     &\leq \delta \eta_0 \|T(\nabla u_i)\| + \frac{\delta^2\eta_0^2}{4} \leq \delta^{\frac{3}{2}} \eta_0 \Gamma_i^{\frac{1}{4}}  + \frac{\delta^2\eta_0^2}{4}
\end{align}
Replacing, \eqref{Equation-4.5} by \eqref{Equation-4.5(2)} in the proof of Theorem~\ref{theorem1.1}, we complete the proof of the first part. In order to prove the second part, it is necessary to note only that for the bi-drifted Cheng-Yau operator, the tensor $T$ is divergence-free and we have $\tr(\nabla T)=0$, hence $T_0=0$. Thus, this complete the proof of Theorem~\ref{Generalization-Duetal-teo}.
\end{proof}

\section*{Acknowledgements} 
The author would like to express their sincere thanks to Jose N. V. Gomes and Xia Changyu for useful comments, discussions and constant encouragement. The author is partially supported by Coordenação de Aperfeiçoamento de Pessoal de Nível Superior (CAPES) in conjunction with Fundação Rondônia de Amparo ao Desenvolvimento das Ações Científicas e Tecnológicas e à Pesquisa do Estado de Rondônia (FAPERO).


\begin{thebibliography}{Referencial}
\bibitem{AraujoFilho}
Araújo Filho, M. C.: Estimates for the first eigenvalues of Bi-drifted Laplacian on smooth metric measure space. Differential Geom. Appl. 80, Paper No. 101839 (2022).
\bibitem{Ashbaugh}
Ashbaugh, M.S.: Isoperimetric and universal inequalities for eigenvalues. In: Davies, E.B., Safalov, Yu. (eds.) Spectral Theory and Geometry (Edingurgh, 1998), London Mathematical Society Lecture Notes, vol. 273, pp. 95–139. Cambridge University Press, Cambridge (1999).
\bibitem{AshbaughBenguria}
 Ashbaugh, M., Benguria, R.:  On Rayleigh's conjecture for the clamped plate and its generalization to three dimensions. Duke Math. J. 78, no. 1, 1-17 (1995).
 \bibitem{AshbaughLaugesen}
 Ashbaugh, M., Laugesen, R.S.:  Fundamental tones and buckling loads of clamped plates. Ann. Scuola Norm. Sup. Pisa Cl. Sci. (4) 23, no. 2, 383-402 (1996).
\bibitem{Cheng-etal1}
 Cheng, Q.M.,  Huang, G.  Wei, G.: Estimates for lower order eigenvalues of a clamped plate problem. Calc. Var. PDE. 38, 409–416 (2010)
\bibitem{Cheng-etal}
Cheng, Q.M., Ichikawa, T., Mametsuka, S.:  Estimates for eigenvalues of a clamped plate problem on Riemannian manifolds. J. Math. Soc. Jpn. 62, 673–686 (2010)
\bibitem{ChengYang5}
Cheng, Q.M., Yang, H.C.: Inequalities for eigenvalues of a clamped plate problem. Trans. Am. Math. Soc. 358, 2625–2635 (2006)  
\bibitem{ChengYau}
Cheng, S.Y.,  Yau, S.T.: Hypersurfaces with constant scalar curvature, Math. Ann. 225, 195-204 (1977)
\bibitem{doCarmoWangXia}
do Carmo, P. Wang, Q. L. and Xia, C. Y.: Inequalities for eigenvalues of elliptic operators in divergence form on Riemannian manifolds, Ann. Mat. Pur. Appl. 189, 643–660 (2010)
\bibitem{Du-etal}
Du, F. Wu, C. Li, G. and  Xia, C.: Estimates for eigenvalues of the bi-drifting Laplacian operator, Z. Angew. Math. Phys. 66(3), 703–726 (2015)
\bibitem{GomesAraujoFilho}
Gomes, J.N.V., Araújo Filho, M.C.: Estimates of eigenvalues of an elliptic differential system in divergence form,	arXiv:2011.13507 [math.AP]
\bibitem{GomesMiranda}
Gomes, J.N.V. and Miranda, J.F.R.: Eigenvalue estimates for a class of elliptic differential operators in divergence form, Nonlinear Anal. 176,  1-19 (2018)
\bibitem{GoverOrsted}
 R.Gover, A. Orsted, B.: Universal principles for Kazdan-Warner and Pohozaev-Schoen type identities. Commun. Contemp. Math. 15(4), 1350002, 27 pp. (2013)
\bibitem{Grosjean}
Grosjean, J.F.:  Extrinsic upper bounds for the first eigenvalue of elliptic operators. Hokkaido Math. J. 33(2), 319–339 (2004) 
\bibitem{HileYeh}
Hile, G. N., Yeh, R. Z.: Inequalities for eigenvalues of the biharmonic operator, Pacific J Math 112,  115-133 (1984)
\bibitem{Hook}
Hook, S. M.: Domain independent upper bounds for eigenvalues of elliptic operator. Trans. Am. Math. Soc. 318, 615–642 (1990)  
\bibitem{Jost}
Jost, J., Jost, X.L., Wang, Q., Xia, C.: Universal bounds for eigenvalues of the polyharmonic operators, Trans. Amer. Math. Soc. 363(4), 1821–1854 (2011)
\bibitem{Kristaly} Kristaly, A.: Fundamental tones of clamped plates in non-positively curved spaces. Adv. Math. 367, 107113, 39 pp. (2020).
\bibitem{Nadirashvili}
Nadirashvili, N.S.:  Rayleigh's conjecture on the principal frequency of the clamped plate. Arch. Rational Mech. Anal. 129, no. 1, 1-10 (1995).
\bibitem{PJ-Olver}
Olver, P.J.: Introduction to Partial Differential Equations, Undergraduate Texts in Mathematics, Springer, DOI 10.1007/978-3-319-02099-0.
\bibitem{Payne-etal}
Payne, L.E., Pólya, G., Weinberger, H.F.:  On the ratio of consecutive eigenvalues. J. Math. Phys. 35, 289– 298 (1956) 
\bibitem{Roth2}
Roth, J.: Reilly-type inequalities for Paneitz and Steklov eigenvalues. Potential Anal. 53(3), 773–798 (2020) 
\bibitem{Serre}
Serre, D.: Divergence-free positive symmetric tensors and fluid dynamics, Ann. I.H.Poincaré – AN35, 1209–1234 (2018)  
\bibitem{Wang}
 Wang, Q.M.: Isoparametric functions on Riemannian manifolds I, Math. Ann. 277,  639-646 (1987)
\bibitem{WangXia}
Wang, Q., Xia, C.: Inequalities for eigenvalues of a clamped plate problem, Calc. Var. 40, 273–289 (2011) 
\bibitem{Yang}
Yang, H.C.: An estimate of the difference between consecutive eigenvalues, preprint IC/91/60 of ICTP, Trieste, 1991.
\end{thebibliography}
\end{document}